\documentclass[reqno,11pt]{amsart}
\usepackage{a4wide,color,eucal,enumerate,mathrsfs}
\usepackage[normalem]{ulem}
\usepackage{amsmath,amssymb,epsfig,amsthm} 
\usepackage[latin1]{inputenc}
\usepackage{psfrag}
\usepackage{hyperref,cleveref}
\usepackage{xcolor}

\definecolor{lightseagreen}{rgb}{0.13, 0.7, 0.67}
\definecolor{darkred}{rgb}{0.55, 0.0, 0.0}

\usepackage{pgfplots}
\pgfplotsset{compat=1.18}
\usepgfplotslibrary{polar}
\usetikzlibrary{intersections}
\usetikzlibrary{positioning}

\newcommand{\Ncal}{\mathcal{N}}

\def\dist{{\mathop\mathrm{dist}}}
\def\loc{{\mathop\mathrm{\,loc\,}}}

\def\ez{\epsilon}

\def\bint{{\ifinner\rlap{\bf\kern.35em--}
\int\else\rlap{\bf\kern.45em--}\int\fi}\ignorespaces}

\def\bbint{{\ifinner\rlap{\bf\kern.35em--}
\hspace{0.078cm}\int\else\rlap{\bf\kern.45em--}\int\fi}\ignorespaces}

\def\esssup{{\rm \,esssup\,}}

\def\diam{{\mathop\mathrm{\,diam\,}}}

\def\dfrac{\displaystyle\frac}

\newtheorem{thm}{Theorem}[section]
\newtheorem{lem}[thm]{Lemma}
\newtheorem{prop}[thm]{Proposition}
\newtheorem{defn}[thm]{Definition}
\numberwithin{equation}{section}

\theoremstyle{remark}
\newtheorem{rem}[thm]{Remark}

\def\bint{{\ifinner\rlap{\bf\kern.35em--}
\int\else\rlap{\bf\kern.45em--}\int\fi}\ignorespaces}

\usepackage{graphicx}
\usepackage{import}
\usepackage{xifthen}
\usepackage{pdfpages}
\usepackage{transparent}

\title[An anisotropic Serrin's problem in general domains]{An anisotropic Serrin's problem in general domains}
\author{Alessio Figalli and Yi Ru-Ya Zhang}
\date{\today}

\address{ETH Z\"urich, Department of Mathematics, R\"amistrasse 101, 8092, Z\"urich, Switzerland}
\email{alessio.figalli@math.ethz.ch}  
\address{State Key Laboratory of Mathematical Sciences, Academy of Mathematics and Systems Science, Chinese Academy of Sciences, Beijing 100190, China}
\address{Institute of Mathematics, Academy of Mathematics and Systems Science, the Chinese Academy of Sciences, Beijing 100190, China}
\email{yzhang@amss.ac.cn}

 \thanks{The second author is funded by the National Key R\&D Program of China (Grant No. 2025YFA1018400 \&  No. 2021YFA1003100), NSFC grant No. 12288201 \& No. 12571128, the Chinese Academy of Sciences, and CAS Project for Young Scientists in Basic Research, Grant No. YSBR-031. }


\begin{document}

\begin{abstract}
Serrin's symmetry theorem shows that the classical overdetermined torsion problem forces the
domain to be a ball. Extending this rigidity statement to merely Lipschitz (and more generally
rough) domains in the weak formulation has been a long-standing and challenging problem,
recently resolved by the authors in~\cite{FZ2025}.

In this paper we address the corresponding question in the
anisotropic setting: Given a uniformly convex $C^{2,\gamma}$ anisotropy $H$, we study the
overdetermined problem for the anisotropic Laplacian
$\Delta_H u={\rm div}\big(H(\nabla u)\,DH(\nabla u)\big)$ on a bounded indecomposable set
of finite perimeter $\Omega$. Assuming the Ahlfors--David regularity of $\partial^*\Omega$ and
a global $\beta$-number square-function bound (a weak uniform rectifiability hypothesis),
we prove that a weak solution exists if and only if $\Omega$ is a translate and dilation of
{the reflected Wulff shape $-K$}, in which case the solution is unique and explicit. In particular, the result
applies to Lipschitz domains. 

While our approach follows the rough-domain strategy of~\cite{FZ2025} at a high level, 
the key Laplacian-specific ingredients exploited there have no direct analog for $\Delta_H$, 
necessitating the development of new ideas and techniques.

\end{abstract}
 
\maketitle

\section{Introduction}
Let $H:\mathbb R^n\to[0,\infty)$ be a convex, $1$-homogeneous positive function, smooth away from the origin, and define the dual function $H_*$ as
\[
H_*(x):=\sup\{x\cdot y:\ H(y)=1\},\qquad x\in\mathbb R^n,
\]
and
\[
K:=\{x\in\mathbb R^n:\ H_*(x)<1\}.
\]
For a set $E$ of finite perimeter, we define the anisotropic perimeter corresponding to $H$ by
\[
P_H(E):=\int_{\partial^*E} H(-\nu_E)\,d\mathscr H^{n-1},
\]
where $\nu_E$ is the measure-theoretic outer unit normal of $E$.
Equivalently, this is the anisotropic perimeter with surface tension $F(\nu):=H(-\nu)$; with this definition,
the Wulff minimizer for $P_H$ (namely, the minimizer of $P_H$ at a given volume) is the reflected body $-K$.

Define the energy integrand and the anisotropic Laplacian as
\[
V(\xi):=\frac12 H(\xi)^2,
\qquad 
\Delta_H u := {\rm div}\bigl(DV(\nabla u)\bigr)={\rm div}\bigl(H(\nabla u)\,DH(\nabla u)\bigr),
\]
and consider the following anisotropic overdetermined system:\\
\emph{Find 
$u\in W^{1,2}_0(\Omega)$ such that
\begin{equation*}
\int_{\Omega} DV(\nabla u)\cdot \nabla\varphi \, dx
= -\mathbf{c}\int_{\partial^* \Omega} \varphi \, H(-\nu_x)\, d\mathscr{H}^{n-1}
+ \int_{\Omega} \varphi \, dx,
\qquad \forall\, \varphi\in C^1(\mathbb{R}^n),
\end{equation*}
where $\mathbf{c}=\frac{|\Omega|}{P_H(\Omega)}$.}

In terms of the anisotropic Laplacian,
the system can be equivalently restated in distributional form as
\[
u \in W^{1,2}(\mathbb R^n), \qquad u=0\ \text{a.e.\ in }\mathbb R^n\setminus \Omega,\qquad
\Delta_H u=\mathbf{c}H(-\nu)\,\mathscr{H}^{n-1}\llcorner \partial^*\Omega - \mathbf{1}_{\Omega}\,dx.
\]
The foundational case where $K=B$ (the Euclidean unit ball) and $\Omega$ is of class $C^2$ was first investigated by Serrin
\cite{S1971}, who proved that a solution exists if and only if $\Omega$ is a ball. Later, Weinberger \cite{W1971} provided an alternative proof. Building on Weinberger's approach, several other
methods have been developed; see, for instance, \cite{BNST2008, CH1998, PS1989} as well as the recent manuscript \cite{CWZ2025}.

When both $\Omega$ and $K$ are sufficiently smooth so that $u\in C^2(\overline{\Omega})$ (i.e.\ $u$ is a classical solution),
\cite{CS2009} and \cite{WX2011} independently showed that $\Omega$ must be homothetic to {the Wulff shape associated with the chosen surface tension (with the present convention $H(-\nu)$, this is $-K$)}. These arguments refine the techniques introduced in \cite{BNST2008} and \cite{W1971},
respectively.

\subsection{Serrin's theorem in general domains}

In 1992, Vogel \cite{V1992} showed that if $\Omega$ is a $C^1$ domain admitting a solution to Serrin's overdetermined problem
(for the Laplacian or for certain special degenerate operators), then $\Omega$ must  be of class $C^{2}$ and Serrin's theorem applies.
In his approach Vogel assumed the boundary behavior
\[
u(x)\to 0 \quad \text{and}\quad |\nabla u|(x)\to \mathbf{c} \ \text{ uniformly as } x\to \partial \Omega,
\]
and used the Alt--Caffarelli-type free boundary regularity theory to deduce the $C^2$-regularity of $\Omega$.

Later, in the classical case $K=B$, Berestycki raised the question of what happens if the domain is $C^2$ except at a potential corner, and $u$ is a strong
solution everywhere except at that corner. This was resolved in \cite{P1998} using an adapted moving plane
method that avoids the singular point, sparking interest in extending Serrin's theorem to a broader class of domains. Specifically,
\cite[Question 7.1]{HLL2024} asked the following in the case $K=B$:\\
\emph{Does Serrin's theorem hold if $\Omega$ is merely Lipschitz and $u$ solves the equation under the weak formulation?}

This question was recently answered in \cite{FZ2025} for an even broader class of domains. Recall that a measurable set
$\Omega\subseteq \mathbb R^n$ has \emph{finite perimeter} if the distributional gradient of its characteristic function
$\mathbf{1}_\Omega$ is a $\mathbb R^n$-valued Radon measure $D\mathbf{1}_\Omega$ with finite total variation, i.e.
$|D\mathbf{1}_\Omega|(\mathbb R^n)<\infty$.
By the Lebesgue--Besicovitch differentiation theorem, for $|D\mathbf{1}_{\Omega}|$-a.e.\ $x$,
\[
\lim_{r\to 0^+}\frac{D\mathbf{1}_\Omega(B_r(x))}{|D\mathbf{1}_\Omega|(B_r(x))} = -\nu_x,
\qquad |\nu_x| = 1.
\]
The set of such points is the \emph{reduced boundary} $\partial^*\Omega$, and $\nu_x$ is the measure-theoretic outer unit normal.
By De Giorgi's structure theorem, $\partial^*\Omega$ is  $(n-1)$-rectifiable. Moreover, $\Omega$ can be adjusted on a null
set so that $\overline{\partial^*\Omega}=\partial\Omega${; see \cite[Proposition 12.19]{M2012}. Throughout the sequel, whenever the topological
boundary, distance to the boundary, or boundary neighborhoods are used, we make this representative choice.}

A set $E$ of finite perimeter is \emph{indecomposable} if for any $F \subset E$ of finite perimeter satisfying
\[
\mathscr H^{n-1}(\partial^* E) = \mathscr H^{n-1}(\partial^* F) + \mathscr H^{n-1}(\partial^* (E\setminus F)),
\]
either $|F|=0$ or $|E\setminus F|=0$. For further details, see \cite{ACMM2001} and \cite[Sections 12 \& 15]{M2012}.

In \cite{FZ2025}, it was proven (via geometric measure theory techniques) that if a bounded, indecomposable set of finite
perimeter $\Omega$ satisfies, for some $A>0$, the measure-theoretic condition
\begin{equation}\label{FZ measure density}
\mathscr H^{n-1}(B_r(x)\cap \partial^* \Omega)\le A r^{n-1}
\quad \text{for $\mathscr H^{n-1}$-a.e.\ } x\in \partial^* \Omega \text{ and } r \in (0,1),
\end{equation}
along with
\begin{equation}\label{weak}
u \in W^{1,2}_0(\Omega)\quad\text{and}\quad
\int_{\Omega} \nabla u\cdot \nabla\varphi \, dx
= -\mathbf{c}\int_{\partial^* \Omega} \varphi \, d\mathscr{H}^{n-1} + \int_{\Omega} \varphi \, dx
\quad \forall\, \varphi\in C^1(\mathbb{R}^n),
\end{equation}
then $\Omega$ must be a ball. Notably, Lipschitz domains satisfy \eqref{FZ measure density}, thereby resolving
\cite[Question 7.1]{HLL2024}. 

More recently, \cite{DZ2025,DR2026} presented two alternative approaches for Lipschitz domains: One based on non-tangential limits and harmonic analysis techniques, and another relying on refined regularity theory for weak solutions of Alt--Caffarelli-type free boundary problems. We stress that all three proofs (those in \cite{FZ2025}, \cite{DZ2025}, and \cite{DR2026}) depend in an essential way on structural properties specific to the Laplacian.

The paper \cite{DZ2025} also proposes Conjecture~1.4, formulating an anisotropic Serrin-type conjecture for rough domains. However, the solvability of regularity problems for second-order divergence-form elliptic operators in general Lipschitz domains remains open, except in settings where the local Lipschitz constant is sufficiently small (e.g., in $C^1$ domains; see \cite{DPR2017}). Consequently, the strategy in \cite{DZ2025} to extend Serrin's theorem from the Laplacian to general anisotropic operators in Lipschitz domains faces intrinsic limitations and requires strong additional assumptions on both the geometry of the domain and the structure of the operator.

In the present paper, we instead pursue the geometric-measure-theoretic strategy of~\cite{FZ2025}.
A central difficulty is that several identities and rigidity steps used there rely on the special
structure of the Laplacian; developing replacements for these ingredients is the main
new contribution of this work.

\subsection{Main result}
To formulate our result, we first recall (a variant of) the $\beta$-number introduced by P.\ Jones \cite{P1990}: For each $x\in \partial^*\Omega$ and $r>0$, define
\begin{equation}\label{beta ball}
\beta(x,r):=\inf_{P} r^{1-n}\int_{\partial^*\Omega\cap B_r(x)} \frac{\dist(z,P)}{r}\, d\mathscr H^{n-1}(z),
\end{equation}
where $P$ ranges over all affine hyperplanes intersecting $B_r(x)$.

From now on, when considering a set of finite perimeter $\Omega$, we shall aways consider the representative for which $\partial\Omega=\overline{\partial^*\Omega}$.
Our main theorem states as follows.

\begin{thm}\label{main thm}
Let $\Omega\subset \mathbb R^n$ be a bounded indecomposable set of finite perimeter,  and assume there exist constant $A_1,A_2\geq 1$ such that 
\begin{equation}\label{beta number}
\int_{\partial^* \Omega}\int_{0}^{1} \beta(x,s)^2 \, \frac{ds} {s} \, d\mathscr H^{n-1}(x)\le A_1
\end{equation}
and that, for every $x\in \partial^* \Omega$ and every $r \in (0,1)$,
\begin{equation}\label{domain boundary1}
A_2^{-1}r^{n-1}\le \mathscr H^{n-1}(B_r(x)\cap \partial^* \Omega)\le A_2 r^{n-1}.
\end{equation}
Let $K$ be a bounded uniformly convex body whose boundary is of class
$C^{2,\gamma}$ for some $\gamma>0$ small, and let $H$ be the corresponding Wulff potential.
Then $\Omega$ admits a solution $u\in W^{1,2}(\mathbb R^n)$ to
\begin{equation}\label{eq:weak set finite per}
u=0\quad \text{a.e.\ in }\mathbb R^n\setminus \Omega,\qquad
\Delta_H u=\mathbf{c}H(-\nu)\,\mathscr{H}^{n-1}\llcorner \partial^*\Omega - \mathbf{1}_{\Omega}\,dx,
\end{equation}
in the sense of distributions, if and only if, {up to modifying $\Omega$ on a null set and up to a translation,} 
\[
{
\text{$\Omega$ is homothetic to $-K$}\qquad \text{and}\qquad
u(x)=\dfrac{(r^2-H_*^2(-x))_+}{2n}\quad \text{for some }r>0.
}
\]
\end{thm}

\begin{rem}
Theorem~\ref{main thm} implies, in particular, an anisotropic version of Serrin's theorem for any domain $\Omega$ whose boundary
satisfies \eqref{beta number} and \eqref{domain boundary1} above. Note that Lipschitz domains, or even uniformly $n$-rectifiable sets, do satisfy these assumptions, see e.g.\ \cite{DS1991} and
\cite[Proposition~III.4.2]{DS1993}. See also \cite{T2015, AT2015} for the application of $\beta$-type numbers in characterizing the rectifiability of sets. 

We note that the assumptions on $K$ imply that $H$ is of class $C^{2,\gamma}$ and uniformly convex away from the origin.
\end{rem}

\begin{rem}
We say that a set $\Omega\subset\mathbb R^n$ has finite $(n-1)$-dimensional upper Minkowski content if there exists $M>0$ such that
\begin{equation}\label{eq:minkowski-content}
    {\limsup_{r\to 0^+} \frac{|\mathcal N_{r}(\partial\Omega)|}{r}\le M<\infty,}
\end{equation}
where, for $r>0$, the $r$-neighborhood of $\partial\Omega$ is
\[
\mathcal N_r(\partial\Omega):=\left\{ x\in \mathbb R^n:\ \dist(x,\partial \Omega)\le r\,\diam(\Omega) \right\}.
\]
Note that, via a Besicovitch covering, one can directly conclude from \eqref{domain boundary1} that $\partial \Omega$ has a finite
Minkowski content with $M=M(n,A_2)$.
\end{rem}

\begin{rem}
When there are finitely many cusps along $\partial \Omega$, our method can still be applied, even though the lower bound in the
Ahlfors regularity assumption \eqref{domain boundary1} does not necessarily hold. Indeed, this lower bound is only used to
control the number of ``bad'' balls (see the last conclusion in Proposition~\ref{prop:ADR} and its application in the proof of
Lemma~\ref{second derivative}). Thus, when there are finitely many cusps, one can directly cut off a suitable neighborhood of the cusps in the argument.
\end{rem}

\medskip
The manuscript is organized as follows.
In Section~2, we establish the volume identity \eqref{volume} by following Weinberger's original approach. The primary challenge
arises from the nonlinearity of the operator, which restricts the regularity of the solutions: the original proof relies on global
higher-order differentiability properties that are not guaranteed in this setting. To be specific, if $u\in W^{2,2}(\Omega)$
were known, the proof could be adapted from \cite{FZ2025} with minor technical adjustments. However, while local $W^{2,2}$
regularity is available in our setup, its global counterpart remains unclear, even in Lipschitz domains.

To address this issue, we exploit the geometric properties encoded by the $\beta$-number \eqref{beta number}
(see Proposition~\ref{prop:ADR} in the appendix). In addition, we prove a vanishing property for the Hessian of $u$
(see Item~(5) in Lemma~\ref{basic property u}). These results enable us to prove the validity of  the crucial estimate \eqref{change derivative},
which would follow directly from the chain rule under the assumption $u\in W^{2,2}(\Omega)$.

In Section~3, we analyze   Green's function for the linearized operator and conclude Theorem~\ref{main thm} via classical
$P$-function theory.

\medskip
{\noindent \it Acknowledgments.}
The second author would like to express his gratitude to the Forschungsinstitut f\"ur Mathematik (FIM) at ETH Z\"urich for their warm hospitality and support during his visit, where this work was completed.

\section{Preliminary results}
\label{sec:prelim}

Throughout the manuscript, we denote by $\mathscr H^{n-1}$ the $(n-1)$-dimensional Hausdorff measure, and by $\#$
the counting measure. We also write $\omega_{n-1}$ for the $(n-1)$-dimensional measure of the unit ball in $\mathbb R^{n-1}$. {We employ the notation $a \sim b$ to succinctly denote that}
$$
{C^{-1} a \leq b \leq C a}
$$
{for some absolute constant $C \geq 1$.}

In this section, we establish regularity properties of the weak solution $u$ and of the set $\Omega$, and prove
a volume identity that will be used in the proof of Theorem~\ref{main thm}.

In the next lemma, $\mathring\Omega$ denotes the (topological) interior of $\Omega$. The lemma is a variant of
\cite[Lemma~2.1]{FZ2025}. Here, we use nonlinear potential estimates (see \cite{KM12,KM1994}) to obtain Lipschitz
regularity up to the boundary. We will also need a vanishing property for $D^2u$ at small scales.

\begin{lem}\label{basic property u}
Let $\Omega\subset \mathbb R^n$ be a bounded set of finite perimeter satisfying \eqref{domain boundary1},
and let $u \in W^{1,2}(\mathbb R^n)$ satisfy \eqref{eq:weak set finite per}, where $H$ satisfies the assumptions in Theorem~\ref{main thm}.
Then:
\begin{enumerate}
\item[(1)] $u$ is nonnegative and $\mathring\Omega=\{u>0\}$.

\item[(2)] $u$ is {Lipschitz continuous} (globally in $\mathbb R^n$).

\item[(3)] $\Omega=\mathring\Omega$ up to a set of measure zero. In particular, without loss of generality, we may assume
$\Omega$ to be open.

\item[(4)] For every $x\in \partial^*\Omega$ it holds
\[
\frac{u(x+rz)}{r} \to  \mathbf{a}(x)\bigl(-\nu_x\cdot z \bigr)_+ \quad \text{ as } \ r\to 0,
\]
where $\nu_x$ denotes the measure-theoretic outer unit normal at $x$ and $\mathbf {a}(x):=\frac{\mathbf c}{H(-\nu_x)}$.
{Moreover, the interior gradients have the approximate trace
\[
\nabla u(y)\to -\mathbf a(x)\nu_x
\quad\text{as } y\to x \text{ from inside }\Omega
\]
in the measure-theoretic sense, i.e. after rescaling at $x$ the gradients converge in measure on interior half-balls.}

\item[(5)] For any $0<\kappa<1$, any $\delta>0$, any $\eta>0$, and every $y\in \partial^*\Omega$, there exists
$r_0=r_0(y,n,\kappa,\delta,\eta)>0$ such that for every $0<r<r_0$ one has
\[
\beta(y,r)\le \eta,
\qquad
\frac{\mathscr {H}^{n-1}(\partial^*\Omega\cap B(y,r))}{\omega_{n-1}r^{n-1}}\in (1-\eta,\,1+\eta),
\]
and
\[
\int_{\hat B_{y,\,\kappa,\, r}}|D^2u|\, dx\le \delta \,(\kappa^{-1} r)^{n-1},
\]
where
\[
\hat B_{y,\,\kappa,\, r}:=
B_{\frac{\kappa^{-1}-\kappa}{2}r}\!\left(y-\frac{\kappa^{-1}+\kappa}{2}r\,\nu_y\right){\subset\subset \Omega}.
\]
\end{enumerate}
\end{lem}

\begin{proof}
Since $\Omega$ satisfies \eqref{domain boundary1}, the measure
\[
\mu:=\mathbf{c}H(-\nu)\,\mathscr{H}^{n-1}\llcorner \partial^*\Omega - \mathbf{1}_{\Omega}\,dx
\]
satisfies
\begin{equation}\label{eq:mu n1}
|\mu|(B_r(z))\leq C(n,A_2,H,\mathbf{c})\,r^{n-1}\qquad \forall\,z \in \mathbb R^n,\ \forall\,r \in (0,1).
\end{equation}
Indeed, the boundary part is controlled by \eqref{domain boundary1} and $H$ is bounded on $\mathbb S^{n-1}$, while the volume
part satisfies $|B_r|\lesssim r^n\le r^{n-1}$ for $r<1$.

Hence $u$ solves $\Delta_H u=\mu$ in $\mathbb R^n$, and by \cite[Equation (1.29)]{KM12} (see also \cite{KM1994}) it follows that
$u$ is H\"older continuous.

\medskip
\noindent\textit{Step 1: Nonnegativity and identification of $\{u>0\}$.}
Since $u=0$ a.e.\ in $\mathbb R^n\setminus\Omega$ and $u$ is continuous, we have $u\equiv 0$ on $\mathbb R^n\setminus\mathring\Omega$.
Moreover, on $\mathring\Omega$ the equation reduces to $\Delta_Hu=-1<0$ in the weak sense. Therefore, by the strong maximum
principle, either $u\equiv 0$ on a connected component of $\mathring\Omega$ or $u>0$ there. Since $u\not\equiv 0$ in $\Omega$
(because $\Delta_Hu=-1$ in $\mathring\Omega$), we conclude $u\ge 0$ and
\[
\mathring\Omega=\{u>0\}.
\]
This proves (1).

\medskip
\noindent\textit{Step 2: Boundary linear growth.}
We claim that there exists $C>0$ such that
\begin{equation}\label{eq:Lip bdry}
u(y)\le C\,\dist(y,\partial\Omega)\qquad \text{for all }y\in \Omega.
\end{equation}
Assume by contradiction that \eqref{eq:Lip bdry} fails. {Then, for $k\gg 1$, we can find
$x_k\in \partial\Omega$ and $0<r_k\downarrow0$ such that
\[
u(x_k+z)\le k|z|\quad\text{for }z\in B_1\setminus B_{r_k},
\qquad
\max_{z\in \overline B_{r_k}}u(x_k+z)=kr_k.
\]
Equivalently, after the rescaling below,
\[
\sup_{z\in B_{1/r_k}}\frac{u(x_k+r_k z)}{r_k\max\{1,|z|\}}=k.
\]}
Define
\[
v_k(x):=\frac{u(x_k +r_k x)}{kr_k}.
\]
Then $v_k\ge 0$, $v_k(0)=0$ since $u(x_k)=0$, and $\max_{\overline B_1}v_k = 1$. Moreover,
\[
v_k(y)\le \max\{1,|y|\}\qquad \forall\,y \in B_{1/r_k}.
\]
A scaling computation, together with \eqref{eq:mu n1}, yields that $v_k$ solves
\[
\Delta_H v_k=\nu_k \quad\text{in }\mathbb R^n,
\]
where the measures $\nu_k$ satisfy
\[
|\nu_k|(B_r(z))\leq \frac{C(n,A_2,H,\mathbf{c})}{k}\,r^{n-1}\qquad \forall\,z\in\mathbb R^n,\ \forall\,r\in(0,1).
\]
Letting $k \to \infty$, by \cite{BM92} and the local H\"older compactness (again from \cite{KM12}, thanks to the bounds on $\nu_k$),
we find (up to a subsequence) a limit function $v_\infty:\mathbb R^n\to \mathbb R$ such that
\[
\Delta_H v_\infty =0,\qquad v_\infty \ge 0,\qquad v_\infty(0)=0,\qquad \max_{\overline B_1}v_\infty = 1.
\]
This is impossible since, by the strong maximum principle, the first three properties above imply that $v_\infty$ must be identically zero. Hence \eqref{eq:Lip bdry} holds.

\medskip
\noindent\textit{Step 3: Global Lipschitz continuity.}
Since $\Delta_H u=-1$ inside $\mathring\Omega$ and $H\in C^{1,1}(\mathbb R^n\setminus\{0\})$, the boundary growth estimate
\eqref{eq:Lip bdry} combined with interior gradient estimates implies that $u$ is uniformly Lipschitz in $\mathring\Omega$.
Recalling that $u$ vanishes in $\mathbb R^n\setminus \Omega$, we conclude that $u$ is globally Lipschitz in $\mathbb R^n$.
This proves (2).

\medskip
\noindent\textit{Step 4: $\Omega=\mathring\Omega$ up to a null set.}
Since $u \ge 0$, the measure $\Delta_H u$ is nonnegative in the set $\{u=0\}=\mathbb R^n\setminus \mathring\Omega$, while
inside $\mathring\Omega=\{u>0\}$ we have $\Delta_H u=-1$. Thus
\[
\Delta_H u= \mu_+- \mathbf{1}_{\mathring\Omega}\,dx
\]
for some nonnegative measure $\mu_+$. Comparing this decomposition with \eqref{eq:weak set finite per} yields
\[
\mu_+=\mathbf{c}H(-\nu)\,\mathscr{H}^{n-1}\llcorner \partial^*\Omega
\qquad\text{and}\qquad
\mathbf{1}_{\Omega}\,dx=\mathbf{1}_{\mathring\Omega}\,dx.
\]
This implies that $\Omega$ and $\mathring\Omega$ coincide a.e., thus proving (3).

\medskip
\noindent\textit{Step 5: Boundary blow-up.}
Fix $x\in \partial^*\Omega$ and consider the blow-ups
\[
v_r(z)=\frac{u(x+rz)}{r},\qquad \Omega_{x,r}:=\frac{\Omega-x}{r}.
\]
By the $1$-homogeneity of $H$, a scaling computation gives
\[
\Delta_H v_r=\mathbf{c}H(-\nu_{\Omega_{x,r}})\,\mathscr{H}^{n-1}\llcorner \partial^* \Omega_{x,r}
- r\,\mathbf{1}_{\Omega_{x,r}}\,dz.
\]
Since $x\in\partial^*\Omega$, we have $\Omega_{x,r}\to \mathcal H_x$ in $L^1_{\rm loc}$, where $\mathcal H_x$ is the half-space
with outer unit normal $\nu_x$. Moreover, by the Lipschitz bound, $\{v_r\}$ is locally equi-Lipschitz, hence along any sequence
$r_i\downarrow 0$ we may extract a subsequence such that $v_{r_i}\to v_0$ locally uniformly. The limit satisfies
\[
\Delta_H v_0=0\quad \text{in }\mathcal H_x,\qquad v_0=0\quad \text{in }\mathbb R^n\setminus \mathcal H_x,\qquad v_0\ge 0.
\]
By the anisotropic Liouville classification of such global half-space solutions, \footnote{{One may prove this directly in the half-space: tangential difference quotients solve the linearized uniformly elliptic equation with zero boundary trace, and the boundary Harnack/Liouville argument forces them to vanish. Thus the solution depends only on the normal variable, where the one-dimensional computation gives the stated affine profile.}} there exists $\mathbf a(x)\ge 0$ such that
\[
v_0(z)=\mathbf a(x)\,(-\nu_x\cdot z)_+.
\]
On the other hand, again using $x\in\partial^*\Omega$, one has
\[
H(-\nu_{\Omega_{x,r}})\,\mathscr{H}^{n-1}\llcorner \partial^*\Omega_{x,r}
\rightharpoonup
H(-\nu_{\mathcal H_x})\,\mathscr{H}^{n-1}\llcorner \partial\mathcal H_x,
\]
so passing to the limit in the equation gives
\[
\Delta_H v_0=\mathbf{c}H(-\nu_{\mathcal H_x})\,\mathscr{H}^{n-1}\llcorner \partial\mathcal H_x.
\]
Comparing with $v_0=\mathbf a(x)(-\nu_x\cdot z)_+$ yields
\[
\mathbf a(x)\,H(-\nu_x)=\mathbf c.
\]
{Finally, the same compactness gives the stated trace of the gradient. Indeed, on every compact subset of
$\mathcal H_x$ the functions $v_r$ solve uniformly elliptic equations with right-hand side tending to zero and converge
uniformly to the affine profile; interior Caccioppoli estimates therefore give
$\nabla v_r\to -\mathbf a(x)\nu_x$ in $L^1_{\rm loc}(\mathcal H_x)$. The part of
$B_1\cap\Omega_{x,r}$ within distance $\tau$ of the limiting hyperplane has measure $O(\tau)$, uniformly for $r$ small,
and the gradients are uniformly bounded. Letting first $r\downarrow0$ and then $\tau\downarrow0$ proves the
measure-theoretic interior trace of $\nabla u$.}
This proves (4).

\medskip
\noindent\textit{Step 6: Proof of (5).}
Fix $y\in\partial^*\Omega$, $\kappa\in(0,1)$, $\delta>0$, and $\eta>0$.

\smallskip
\noindent\emph{(i) $\beta$-flatness and density.}
Since $y\in\partial^*\Omega$, by De Giorgi's structure theorem \cite[Theorem~15.9]{M2012} we have
\[
\lim_{r\downarrow 0}\beta(y,r)=0,
\qquad
\lim_{r\downarrow 0}\frac{\mathscr H^{n-1}(\partial^*\Omega\cap B(y,r))}{\omega_{n-1}r^{n-1}}=1.
\]
Hence there exists $r_1=r_1(y,\eta)>0$ such that for all $0<r<r_1$ the first two inequalities in (5) hold.

\smallskip
\noindent\emph{(ii) Hessian smallness on the interior ball.}
Consider the blow-ups
\[
u_{y,r}(z):=\frac{u(y+r z)}{r}.
\]
By Step~5 (applied at $y$), we have $u_{y,r}\to \mathbf a(y)(-\nu_y\cdot z)_+$ locally uniformly as $r\downarrow 0$.
{The ball}
\[
\hat B_{0,\kappa,1}:=
B_{\frac{\kappa^{-1}-\kappa}{2}}\!\left(-\frac{\kappa^{-1}+\kappa}{2}\,\nu_y\right)
\]
{is compactly contained in $\mathcal H_y$, and the affine limit
$\mathbf a(y)(-\nu_y\cdot z)_+$ is bounded from below by a positive constant on this ball. Since
$u_{y,r}\to \mathbf a(y)(-\nu_y\cdot z)_+$ locally uniformly and
$\Omega_{y,r}=\{u_{y,r}>0\}$ up to null sets, for $r$ small enough the same ball is compactly contained in
$\Omega_{y,r}$. Scaling back, this is exactly $\hat B_{y,\kappa,r}\subset\subset\Omega$.}

On $\hat B_{0,\kappa,1}$ the functions $u_{y,r}$ solve
\[
{\rm div}\bigl(DV(\nabla u_{y,r})\bigr)=-r
\]
and are uniformly Lipschitz (since $\nabla u_{y,r}(z)=\nabla u(y+r z)$ and $u$ is Lipschitz). Hence interior $W^{2,2}$
estimates for uniformly elliptic quasilinear equations yield a uniform bound
\[
\|D^2u_{y,r}\|_{L^2(\hat B_{0,\kappa,1})}\le C(n,\kappa,L),
\]
independent of $r$ (for $r$ sufficiently small). Since the limit profile $\mathbf a(y)(-\nu_y\cdot z)_+$ is affine on
$\hat B_{0,\kappa,1}$, elliptic compactness implies
\[
D^2u_{y,r}\to 0 \quad\text{strongly in }L^1(\hat B_{0,\kappa,1})\quad\text{as }r\downarrow 0.
\]
In particular, there exists $r_2=r_2(y,n,\kappa,\delta)>0$ such that for all $0<r<r_2$,
\[
\int_{\hat B_{0,\kappa,1}} |D^2u_{y,r}(z)|\,dz \le \delta\,\kappa^{-(n-1)}.
\]
Scaling back (recall $D^2u(y+r z)=\frac1r D^2u_{y,r}(z)$ and $dx=r^n dz$) gives
\[
\int_{\hat B_{y,\kappa,r}} |D^2u(x)|\,dx
= r^{n-1}\int_{\hat B_{0,\kappa,1}} |D^2u_{y,r}(z)|\,dz
\le \delta\, r^{n-1}\kappa^{-(n-1)} = \delta\,(\kappa^{-1}r)^{n-1}.
\]

\smallskip
Finally, taking $r_0:=\min\{r_1,r_2\}$ yields (5).
\end{proof}

We now establish a key step toward proving the volume identity. The validity of \eqref{change derivative} would follow directly from the chain rule
if one knew $u\in W^{2,2}(\Omega)$ globally;\footnote{Indeed, noticing that $DV(\xi)=D^2V(\xi)\,\xi$, we see that
\begin{multline*}
DV(\nabla u(x))\cdot\lim_{\varepsilon\to 0} \frac{\alpha_{\varepsilon}(x)-\alpha_{-\varepsilon}(x)}{2\varepsilon} = DV(\nabla u(x))\cdot D^2u(x)\cdot x \\
=\nabla u(x) \cdot D^2V(\nabla u(x))\cdot D^2u(x)\cdot x=\nabla u(x) \cdot \lim_{\varepsilon\to 0}\frac{DV(\alpha_{\varepsilon}(x))-DV(\alpha_{-\varepsilon}(x))}{2\varepsilon}\quad \text{a.e.}
\end{multline*}
Thus, if $u\in W^{2,2}(\Omega)$,
\eqref{change derivative} follows from dominated convergence.
} however, since we only have $W^{2,2}_{\rm loc}(\Omega)$ regularity, we need to use a fine localization argument
based on the geometry of $\partial\Omega$. It is in this step that the assumption \eqref{beta number} plays a crucial role.

\begin{lem}\label{second derivative}
Let $\Omega$ be a bounded set of finite perimeter satisfying \eqref{beta number} and \eqref{domain boundary1}, and let
$u$ be a weak solution of \eqref{eq:weak set finite per}. Then
\begin{multline}\label{change derivative}
\liminf_{\varepsilon\to 0}\bigg|
\int_{\Omega} DV(\nabla u(x))\cdot \frac{\alpha_{\varepsilon}(x)-\alpha_{-\varepsilon}(x)}{2\varepsilon}\,dx \\
-
\int_{\Omega} \nabla u(x)\cdot \frac{DV(\alpha_{\varepsilon}(x))-DV(\alpha_{-\varepsilon}(x))}{2\varepsilon}\,dx
\bigg|=0,
\end{multline}
where $\alpha_{\pm\varepsilon}(x):=\nabla u((1\pm\varepsilon)x)$.
\end{lem}

\begin{proof}
Recall that $u\in W^{2,2}_{\rm loc}(\Omega)$ and that $V(\xi)=\frac12 H(\xi)^2$ is $2$-homogeneous. Hence, by Euler's theorem,
\[
DV(\xi)=D^2V(\xi)\,\xi \qquad \text{for }\xi\neq 0.
\]
Up to a translation (fixed once and for all), we may assume that $0\in\Omega$. In particular,
\begin{equation}\label{eq:xbound-diam}
|x|\le \diam(\Omega)\qquad \forall\,x\in\Omega.
\end{equation}
{Since the assertion is invariant under dilations of the independent variable, we also normalize
$\diam(\Omega)=1$ throughout the proof.}
We also recall that $u=0$ a.e.\ on $\mathbb R^n\setminus\Omega$, hence $\nabla u=0$ a.e.\ on $\mathbb R^n\setminus\Omega$.

\medskip
\noindent\textit{Step 1: Algebraic reduction to two symmetric terms.}
Set
\[
A_\varepsilon:=
\int_{\Omega} DV(\nabla u)\cdot \frac{\alpha_{\varepsilon}-\alpha_{-\varepsilon}}{2\varepsilon}\,dx
-
\int_{\Omega} \nabla u\cdot \frac{DV(\alpha_{\varepsilon})-DV(\alpha_{-\varepsilon})}{2\varepsilon}\,dx.
\]
Then
\begin{align*}
A_\varepsilon
&=\int_{\Omega} \Bigg[
DV(\nabla u)\cdot \frac{\alpha_{\varepsilon}-\nabla u}{2\varepsilon}
-\nabla u\cdot \frac{DV(\alpha_{\varepsilon})-DV(\nabla u)}{2\varepsilon}
\Bigg]dx\\
&\quad +\int_{\Omega} \Bigg[
DV(\nabla u)\cdot \frac{\nabla u-\alpha_{-\varepsilon}}{2\varepsilon}
-\nabla u\cdot \frac{DV(\nabla u)-DV(\alpha_{-\varepsilon})}{2\varepsilon}
\Bigg]dx.
\end{align*}
Define
\begin{align*}
I_\varepsilon^\pm
:=\Bigg|\int_{\Omega} \Bigg[
DV(\nabla u(x))\cdot \frac{\alpha_{\pm\varepsilon}(x)-\nabla u(x)}{2\varepsilon}
-\nabla u(x)\cdot \frac{DV(\alpha_{\pm\varepsilon}(x))-DV(\nabla u(x))}{2\varepsilon}
\Bigg]dx\Bigg|.
\end{align*}
Since the second bracket above is the negative of the integrand defining $I_\varepsilon^-$, we have
\[
|A_\varepsilon|\le I_\varepsilon^+ + I_\varepsilon^-.
\]
Therefore, to prove \eqref{change derivative} it suffices to show
\[
\liminf_{\varepsilon\to 0}\big(I_\varepsilon^+ + I_\varepsilon^-\big)=0.
\]
We begin by analyzing $I_\varepsilon^+$.

\medskip
\noindent\textit{Step 2: Restriction to the set $U$ and Taylor expansion.}
Note that
$$\alpha_\varepsilon(x)=\nabla u((1+\varepsilon)x)=0  \quad \text{ for a.e.  } \ x\in \Omega \ \text{ while } \ (1+\ez)x\notin \Omega,$$ 
and one checks directly that
the integrand in $I_\varepsilon^+$ vanishes (the two terms cancel, using $DV(0)=0$).
Thus we may restrict the integration to
\[
U:=\{x\in\Omega:\ (1+\varepsilon)x\in\Omega\}.
\]
On $U$, by the fundamental theorem of calculus,
\[
DV(\alpha_\varepsilon)-DV(\nabla u)
=\int_0^1 D^2V\bigl(\nabla u+t(\alpha_\varepsilon-\nabla u)\bigr)\,(\alpha_\varepsilon-\nabla u)\,dt,
\]
and using $DV(\xi)=D^2V(\xi)\xi$ we obtain
\begin{align*}
I_\varepsilon^+
&=\Bigg|\int_{U}\nabla u(x)\cdot \Bigg[
D^2V(\nabla u(x))
-\int_0^1 D^2V\bigl(\nabla u(x)+t(\alpha_\varepsilon(x)-\nabla u(x))\bigr)\,dt
\Bigg]\frac{\alpha_\varepsilon(x)-\nabla u(x)}{2\varepsilon}\,dx\Bigg|.
\end{align*}

\medskip
\noindent\textit{Step 3: Using the $C^\gamma$ regularity of $D^2V$ on the sphere.}
Let $\gamma_0 \in (0,1]$ be such that $D^2V\in C^{\gamma_0}(\mathbb S^{n-1})$, and fix $\gamma \in (0,\gamma_0)$.
Since $D^2V$ is $0$-homogeneous and $C^\gamma$ on $\mathbb S^{n-1}$, for every $\xi\neq 0$ and every $\zeta\in\mathbb R^n$ we have
\[
|D^2V(\xi)-D^2V(\xi+\zeta)|
\le C\,\bigg(\frac{|\zeta|}{|\xi|}\bigg)^\gamma.
\]
Applying this with $\xi=\nabla u(x)$ and $\zeta=t(\alpha_\varepsilon(x)-\nabla u(x))$, and using $|\nabla u|\le L$ a.e.
(Lemma~\ref{basic property u}(2)), we obtain
\begin{align*}
I_\varepsilon^+
&\le C\int_U |\nabla u(x)|\,\frac{|\alpha_\varepsilon(x)-\nabla u(x)|}{2\varepsilon}\,
\int_0^1 \bigg(\frac{t|\alpha_\varepsilon(x)-\nabla u(x)|}{|\nabla u(x)|}\bigg)^\gamma dt\,dx\\
&\le C(L,\gamma)\int_U \frac{|\alpha_\varepsilon(x)-\nabla u(x)|^{1+\gamma}}{2\varepsilon}\,dx.
\end{align*}
Thus it remains to show that
\begin{equation}\label{eq:goal-alpha}
\int_U \frac{|\nabla u((1+\varepsilon)x)-\nabla u(x)|^{1+\gamma}}{\varepsilon}\,dx\to 0
\qquad\text{along a sequence }\varepsilon\downarrow 0.
\end{equation}

\medskip
\noindent\textit{Step 4: Decomposition into $I_{1,\varepsilon}^++I_{2,\varepsilon}^++I_{3,\varepsilon}^+$.}
Recall that
\[
\mathcal N_{\delta}(\partial \Omega):=\left\{ x\in \mathbb R^n:\ \dist(x,\partial \Omega)\le \delta\,\diam(\Omega) \right\}.
\]
For $t>0$ define
\[
U_t:=\Big\{x\in \Omega\setminus \mathcal N_{t}(\partial\Omega):\ (1+\varepsilon)x\in\Omega\Big\}.
\]
Fix $\kappa\in(0,1)$ (to be sent to $0$ later). We decompose
\begin{align*}
\int_U \frac{|\nabla u((1+\varepsilon)x)-\nabla u(x)|^{1+\gamma}}{2\varepsilon}\,dx
&=\int_{U_{\kappa^{-1}\varepsilon}} \frac{|\nabla u((1+\varepsilon)x)-\nabla u(x)|^{1+\gamma}}{2\varepsilon}\,dx\\
&\quad +\int_{U_{\kappa\varepsilon}\setminus U_{\kappa^{-1}\varepsilon}} \frac{|\nabla u((1+\varepsilon)x)-\nabla u(x)|^{1+\gamma}}{2\varepsilon}\,dx\\
&\quad +\int_{U\setminus U_{\kappa\varepsilon}} \frac{|\nabla u((1+\varepsilon)x)-\nabla u(x)|^{1+\gamma}}{2\varepsilon}\,dx\\
&=: I_{1,\varepsilon}^+ + I_{2,\varepsilon}^+ + I_{3,\varepsilon}^+.
\end{align*}

\medskip
\noindent\textit{Step 5: Estimate of $I_{1,\varepsilon}^+$ (interior region).}
On $U_{\kappa^{-1}\varepsilon}$ we have $\dist(x,\partial\Omega)\ge \kappa^{-1}\varepsilon\,\diam(\Omega)$, hence both $x$ and $(1+\varepsilon)x$ {remain at a strictly positive distance from the boundary}.
Since $u\in W^{2,2}_{\rm loc}(\Omega)$, interior $W^{2,2}$ estimates for the uniformly elliptic equation
${\rm div}(DV(\nabla u))=-1$ imply that for each ball $B_{2r}(z)\subset\Omega$,
\[
\|D^2u\|_{L^2(B_r(z))}\le C\,r^{\frac n2-1}.
\]
By H\"older's inequality (since $1+\gamma<2$), this yields
\begin{equation}\label{eq:L1+g-Hess}
\int_{B_r(z)} |D^2u|^{1+\gamma}\,dx \le C\,r^{n-1-\gamma}\qquad\text{whenever }B_{2r}(z)\subset\Omega.
\end{equation}
Moreover, by the fundamental theorem of calculus along $t\mapsto (1+t)x$ and \eqref{eq:xbound-diam},
\[
\nabla u((1+\varepsilon)x)-\nabla u(x)=\int_0^\varepsilon D^2u((1+t)x)\,x\,dt,
\]
hence
\[
\frac{|\nabla u((1+\varepsilon)x)-\nabla u(x)|^{1+\gamma}}{\varepsilon}
\le C\,\varepsilon^\gamma\int_0^1 |D^2u((1+t\varepsilon)x)|^{1+\gamma}\,dt.
\]
Integrating over $U_{\kappa^{-1}\varepsilon}$ and using a Besicovitch covering argument with balls
$B_{r(x)}(x)$ where $r(x)\sim \dist(x,\partial\Omega)$, together with \eqref{eq:L1+g-Hess}, yields\footnote{Here we note that $\partial\Omega$ has finite $(n{-}1)$-dimensional upper Minkowski content (a consequence of
	\eqref{domain boundary1}), {which together with the bounded overlap in the Besicovitch covering gives the counting estimate}
	\[
	{\#\{B_{r_k}(x_k):\,2^{-m}\le r_k\le 2^{-m+1}\}\le C(n,M)\,2^{m(n-1)},}
	\] 
	{therefore}
	$$
	{\sum_k r_k^{n-1-\gamma} \leq C (\kappa^{-1}\varepsilon)^{-\gamma}.}
	$$
} 
\[
I_{1,\ez}^+ \le C(L,n,\diam(\Omega))\,\varepsilon^\gamma \sum_k r_k^{n-1-\gamma}
\le C(L,M,n,\diam(\Omega))\,\kappa^\gamma{,}
\]
{where $M$ is the constant corresponding to the Minkowski content \eqref{eq:minkowski-content}.} 

\medskip
\noindent\textit{Step 6: Estimate of $I_{3,\varepsilon}^+$ (very near the boundary).}
Since $|\nabla u|\le L$ a.e., we have
\[
I_{3,\varepsilon}^+\le C(L)\,\varepsilon^{-1}\,|\mathcal N_{\kappa\varepsilon}(\partial\Omega)|
=O(\kappa),
\]
using again the finite upper Minkowski content: $|\mathcal N_{\rho}(\partial\Omega)|=O(\rho)$ as $\rho\to 0$.

\medskip
\noindent\textit{Step 7: Estimate of $I^+_{2,\varepsilon}$ (intermediate region).}
We first split according to whether $(1+\varepsilon)x$ falls within the thin boundary layer.
By the finite upper Minkowski content, we have
\[
\bigl|\{x\in U_{\kappa\varepsilon}\setminus U_{\kappa^{-1}\varepsilon}:\ (1+\varepsilon)x\in \mathcal N_{\kappa\varepsilon}(\partial\Omega)\}\bigr|
=O(|\mathcal N_{\kappa\varepsilon}(\partial\Omega)|)=O(\kappa\varepsilon),
\]
hence, using $|\nabla u|\le L$,
\begin{equation}\label{2.7}
\int_{\{x\in U_{\kappa\varepsilon}\setminus U_{\kappa^{-1}\varepsilon}:\ (1+\varepsilon)x\in \mathcal N_{\kappa\varepsilon}(\partial\Omega)\}}
\frac{|\nabla u((1+\varepsilon)x)-\nabla u(x)|^{1+\gamma}}{2\varepsilon}\,dx
\le O(\kappa).
\end{equation}
We now consider the remaining points
\[
E:=\Bigl(U_{\kappa\varepsilon}\setminus U_{\kappa^{-1}\varepsilon}\Bigr)\cap\Bigl\{x:\ (1+\varepsilon)x\notin \mathcal N_{\kappa\varepsilon}(\partial\Omega)\Bigr\}.
\]

\smallskip
\noindent\emph{Uniform Hessian smallness on a large subset of $\partial^*\Omega$.}
{Fix $\kappa\in(0,1)$. Let $\bar\kappa=\bar\kappa(\kappa)\in(0,\kappa)$ be the geometric constant provided by
Lemma~\ref{lem:flat-cover}, and choose $\delta_\kappa>0$ so small that}
\begin{equation}\label{eq:deltakappa}
{
C(n,\gamma,A_2)\,\delta_\kappa^{1+\gamma}\,\bar\kappa^{-n}\le \kappa^\gamma .
}
\end{equation}
{Fix $\eta\in(0,1)$ and apply Lemma~\ref{basic property u}(5) with the parameters $\bar\kappa$ and
$\delta_\kappa$. For $\mathscr H^{n-1}$-a.e.\ $y\in\partial^*\Omega$ there exists $r_0(y)>0$ such that for all
$0<r<r_0(y)$,}
\begin{equation}\label{three inequ}
{
\beta(y,r)\le\eta,\qquad 
\frac{\mathscr H^{n-1}(\partial^*\Omega\cap B(y,r))}{\omega_{n-1}r^{n-1}}\in(1-\eta,1+\eta),
\qquad
\int_{\widehat B_{y,\bar\kappa,r}}|D^2u|\,dx \le \delta_\kappa\,(\bar\kappa^{-1}r)^{n-1},
}
\end{equation}
{where $\widehat B_{y,\bar\kappa,r}$ denotes the ball $\hat B_{y,\bar\kappa,r}$ of Lemma~\ref{basic property u}(5).}

{Fix once and for all a function $\theta$ as in Proposition~\ref{prop:ADR}, for instance $\theta(\eta)=\eta$.}
For $m\in\mathbb N$ define
\[
G_m:=\Bigl\{y\in\partial^*\Omega:\ \text{the three inequalities in \eqref{three inequ} hold for all }0<r\le 2^{-m}\Bigr\}.
\]
Since $G_m\uparrow \partial^*\Omega$ up to a null set, we may choose $m=m(\kappa,\eta)$ so large that
\[
\mathscr H^{n-1}(\partial^*\Omega\setminus G_m)\le \theta(\eta)\,\eta^{n-1}.
\]
Set
\[
G_{\kappa,\eta}:=G_m,\qquad r_{\kappa,\eta}:=2^{-m}.
\]
In particular, for every $y\in G_{\kappa,\eta}$ and every $0<r\le r_{\kappa,\eta}$ we have
\begin{equation}\label{7.1}
{
\int_{\widehat B_{y,\bar\kappa,r}}|D^2u|\,dx \le \delta_\kappa\,(\bar\kappa^{-1}r)^{n-1}.
}
\end{equation}

\smallskip
\noindent\emph{Good balls and the estimate on the good part.}
{We choose $\eta\le \eta_0(n,\kappa)$, where $\eta_0$ is given by Lemma~\ref{lem:flat-cover}, and apply
Proposition~\ref{prop:ADR} with the auxiliary set $G_\eta=G_{\kappa,\eta}$. Let $\{r_j\}$ and the coverings
$\mathcal B_{r_j}=\mathcal G_{r_j}\cup\mathcal R_{r_j}$ be the resulting sequence and good/bad decomposition.
Let $c_\kappa$ be the constant from Lemma~\ref{lem:flat-cover}. In the rest of the proof we restrict to
$\varepsilon=r_j\le c_\kappa r_{\kappa,\eta}$; this only discards finitely many selected scales.}

{For $B(y_h,\varepsilon)\in\mathcal G_\varepsilon$, set}
\[
{
\widehat B_h:=\widehat B_{y_h,\bar\kappa,\varepsilon}\subset\subset\Omega,\qquad
E_h:=\{x\in E:\ x\in \widehat B_h,\ (1+\varepsilon)x\in \widehat B_h\}.
}
\]
{By Lemma~\ref{lem:flat-cover}, applied with $r=\varepsilon$ (so that the present set $E$ is $E_{\kappa,r}$),}
\[
{
E\setminus\bigcup_{B(y_h,\varepsilon)\in\mathcal G_\varepsilon}E_h
\subset
\bigcup_{B(y_h,\varepsilon)\in\mathcal R_\varepsilon}B(y_h,C_\kappa\varepsilon).
}
\]
		{For $x\in E_h$, letting $(\nabla u)_{\widehat B_h}:=\frac{1}{|\widehat B_h|}\int_{\widehat B_h}\nabla u$,
		we have}
\[
|\nabla u((1+\varepsilon)x)-\nabla u(x)|^{1+\gamma}
\le C_\gamma\Bigl(|\nabla u((1+\varepsilon)x)-(\nabla u)_{\widehat B_h}|^{1+\gamma}
+|\nabla u(x)-(\nabla u)_{\widehat B_h}|^{1+\gamma}\Bigr).
\]
Integrating over $E_h$ and using the change of variables $z=(1+\varepsilon)x$ in the first term
(note that $z\in\widehat B_h$ for $x\in E_h$) yields
\begin{equation}\label{7.2}
\int_{E_h}\frac{|\nabla u((1+\varepsilon)x)-\nabla u(x)|^{1+\gamma}}{2\varepsilon}\,dx
\le C\,\varepsilon^{-1}\int_{\widehat B_h}|\nabla u-(\nabla u)_{\widehat B_h}|^{1+\gamma}\,dx.
\end{equation}
By Sobolev--Poincar\'e applied to $\nabla u$ on $\widehat B_h$, and choosing
$\gamma>0$ so small that $1+\gamma\le \frac{n}{n-1}$, we get
\begin{equation}\label{7.3}
\int_{\widehat B_h}|\nabla u-(\nabla u)_{\widehat B_h}|^{1+\gamma}\,dx
\le {C(n,\gamma)\,(\bar\kappa^{-1}\varepsilon)^{1+\gamma}
\left(\int_{\widehat B_h}|D^2u|\,dx\right)^{1+\gamma}(\bar\kappa^{-1}\varepsilon)^{-n\gamma}.}
\end{equation}
Using \eqref{7.1} with $r=\varepsilon$ (since $y_h\in G_{\kappa,\eta}$ and $\varepsilon\le r_{\kappa,\eta}$),
\[
{
\int_{\widehat B_h}|D^2u|\,dx \le \delta_\kappa\,(\bar\kappa^{-1}\varepsilon)^{n-1},
}
\]
and plugging this into \eqref{7.3} and then \eqref{7.2} gives
\[
{
\int_{E_h}\frac{|\nabla u((1+\varepsilon)x)-\nabla u(x)|^{1+\gamma}}{2\varepsilon}\,dx
\le C(n,\gamma)\,\delta_\kappa^{1+\gamma}\,\bar\kappa^{-n}\,\varepsilon^{n-1}.
}
\]
{Summing over $B(y_h,\varepsilon)\in\mathcal G_\varepsilon$} and using the bounded overlap together with
the finiteness of the upper Minkowski content (so that the number of such balls is
{$\lesssim \varepsilon^{-(n-1)}$; see Footnote 4}) and \eqref{eq:deltakappa}, we obtain
\begin{equation}\label{2.8}
\sum_{B(y_h,\varepsilon)\in\mathcal G_\varepsilon}
\int_{E_h}\frac{|\nabla u((1+\varepsilon)x)-\nabla u(x)|^{1+\gamma}}{2\varepsilon}\,dx
\le {C(M,n,\mathrm{diam}(\Omega))\,\kappa^\gamma.}
\end{equation}

\smallskip
\noindent\emph{Bad part and choice of scales.}
{By Lemma~\ref{lem:flat-cover} and Proposition~\ref{prop:ADR}(\ref{item:badcount}), the remaining portion of $E$ is contained in
the $C_\kappa\varepsilon$-enlargement of at most}
\[
{
C(n,A_1,A_2)\,\varepsilon^{-(n-1)}\Bigl(|\log\varepsilon|^{-1}\eta^{-2}+\theta(\eta)\Bigr)
}
\]
{bad balls. We choose the gauge $\theta(\eta)=o_\eta(1)$, hence}
\begin{equation}\label{2.9}
{
\Bigg|E\setminus\bigcup_{B(y_h,\varepsilon)\in\mathcal G_\varepsilon}E_h\Bigg|
\le C(n,A_1,A_2,\kappa)\,\varepsilon
\Bigl(|\log\varepsilon|^{-1}\eta^{-2}+o_\eta(1)\Bigr).
}
\end{equation}
{Since $|\nabla u|\le L$, the contribution of this bad set to $I^+_{2,\varepsilon}$ is bounded by}
$$
{
C(n,A_1,A_2,L,\kappa)\,\bigl(|\log\varepsilon|^{-1}\eta^{-2}+o_\eta(1)\bigr).
}
$$
{Combining \eqref{2.7}, \eqref{2.8}, and \eqref{2.9}, we obtain the following estimate on $I^+_{2,\varepsilon}$
(for $\varepsilon=r_j$)}
$$
{I_{2, r_j}^+\le O(\kappa) +  C\,\kappa^\gamma
 + C(n,A_1,A_2,L,\kappa)\bigl(|\log r_j|^{-1}\eta^{-2}+o_\eta(1)\bigr),}
$$
{where the constant $C$ in the second term is independent of $\kappa,\eta$ and $j$.}

\medskip
\noindent\textit{Step 8: Conclusion.}
Fix $\kappa\in(0,1)$ and {then choose $0<\eta\le\eta_0(n,\kappa)$}. Let $r_{\kappa,\eta}>0$ be the scale produced in
Step~7 (so that the estimate \eqref{7.1} is valid for all
$0<\varepsilon\le r_{\kappa,\eta}$), {and let $c_\kappa$ be the geometric constant from Lemma~\ref{lem:flat-cover}}.
Let $\{r_j\}_{j\ge 1}$ be the sequence given in Step~7 by Proposition~\ref{prop:ADR} {with auxiliary set $G_{\kappa,\eta}$}.
Since $r_j\downarrow 0$, by discarding finitely many indices we may assume that
{$r_j\le c_\kappa r_{\kappa,\eta}$} for all $j$.

For $\varepsilon=r_j$, combining the bounds in Steps~5--7 yields
\[
I^+_{r_j}\le C\,\kappa^\gamma + O(\kappa)
+ {C(n,A_1,A_2,L,\kappa)}\Bigl(|\log r_j|^{-1}\eta^{-2}+o_\eta(1)\Bigr),
\]
and the same estimate holds for $I^-_{r_j}$. Hence, for $\varepsilon=r_j$,
\[
I^+_{r_j}+I^-_{r_j}\le C\,\kappa^\gamma + O(\kappa)
+ {C(n,A_1,A_2,L,\kappa)}\Bigl(|\log r_j|^{-1}\eta^{-2}+o_\eta(1)\Bigr).
\]
Letting $j\to\infty$ (so that $r_j\downarrow 0$ and $|\log r_j|^{-1}\to 0$), we obtain
\[
\limsup_{j\to\infty}\,\bigl(I^+_{r_j}+I^-_{r_j}\bigr)
\le C\,\kappa^\gamma + O(\kappa) + {C(n,A_1,A_2,L,\kappa)}\,o_\eta(1).
\]
Finally, we first let $\eta\downarrow 0$ (so that $o_\eta(1)\to 0$), and then let
$\kappa\downarrow 0$, to conclude
\[
\limsup_{j\to\infty}\, \bigl(I^+_{r_j}+I^-_{r_j}\bigr)=0.
\]
Since $I^+_\varepsilon+I^-_\varepsilon\ge 0$, this implies
\[
\liminf_{\varepsilon\to 0}\bigl(I^+_\varepsilon+I^-_\varepsilon\bigr)=0.
\]
As explained in Step~1, this yields \eqref{change derivative} and concludes the proof.
\end{proof}

We now establish the following volume identity, which is a cornerstone of our proof. In the Euclidean case, as shown in
\cite{W1971}, this identity can be derived via Pohozaev's method. For analogous results in the smooth anisotropic setting, we refer to
\cite[Lemma~4.2]{CS2009} and \cite[Theorem~5]{WX2011}. Our approach follows \cite[Lemma~2.2]{FZ2025}, with additional care due to
the nonlinearity of the operator.

\begin{lem}\label{lem:volume}
Let $\Omega\subset \mathbb R^n$ be a bounded set of finite perimeter satisfying {\eqref{beta number} and} \eqref{domain boundary1},
and let $u \in W^{1,2}(\mathbb R^n)$ satisfy \eqref{eq:weak set finite per}. Then
\begin{equation}\label{volume}
(n+2)\int_{\Omega}u\,dx =  {\mathbf{c}^2 n}\, |\Omega|.
\end{equation}
\end{lem}

\begin{proof}
By Lemma~\ref{basic property u}(3), we may assume that $\Omega$ is open.
Up to a translation (fixed once and for all), we also assume $0\in\Omega$, so that $|x|\le \diam(\Omega)$ for all $x\in\Omega$.

\medskip
\noindent\textit{Step 1: Testing with the dilation difference quotient.}
Given $\varepsilon>0$, set
\begin{equation}\label{eq:varphi eps}
\varphi_\varepsilon(x)= \frac{u((1+\varepsilon) x)-u((1-\varepsilon)x)}{2\varepsilon}.
\end{equation}
Since $u$ is Lipschitz (Lemma~\ref{basic property u}(2)) {for some Lipschitz constant $L>0$}, $\varphi_\varepsilon$ is Lipschitz for $\varepsilon>0$ fixed and satisfies
\[
{|\varphi_\varepsilon(x)|\le L \diam(\Omega)\qquad \forall\,x\in\mathbb R^n,\ \forall\,\varepsilon>0.}
\]
By approximation (mollifying $\varphi_\varepsilon$), we may test \eqref{eq:weak set finite per} against $\varphi_\varepsilon$ and obtain
\begin{equation}\label{eq:test vol}
-\mathbf{c} \int_{\partial^*\Omega}  \varphi_\varepsilon \, H(-\nu_x) \, d\mathscr H^{n-1}
+\int_{\Omega} \varphi_\varepsilon\, dx
=  \int_{\Omega} DV(\nabla u)\cdot \nabla\varphi_\varepsilon \, dx .
\end{equation}

\medskip
\noindent\textit{Step 2: The term $\int_\Omega \varphi_\varepsilon$.}
For a.e.\ $x\in\Omega$ we have $\varphi_\varepsilon(x)\to \nabla u(x)\cdot x$ as $\varepsilon\to 0$. Hence, by dominated convergence,
\[
\int_{\Omega} \varphi_\varepsilon\, dx \to \int_{\Omega} \nabla u\cdot x \, dx.
\]
Since $u\in W^{1,2}_0(\Omega)$ and ${\rm div}(x)=n$, an integration by parts yields
\begin{equation}\label{left 1}
\int_{\Omega} \nabla u\cdot x \, dx=-\int_\Omega u\,{\rm div}(x)\,dx=-n\int_{\Omega} u\, dx,
\qquad\text{so}\qquad
\int_{\Omega} \varphi_\varepsilon\, dx \to -n\int_{\Omega} u\, dx .
\end{equation}

\medskip
\noindent\textit{Step 3: The boundary term.}
Fix $x\in\partial^*\Omega$. Since $u=0$ a.e.\ on $\mathbb R^n\setminus\Omega$ and $u$ is Lipschitz, the boundary blow-up
from Lemma~\ref{basic property u}(4) gives, for any fixed $z\in\mathbb R^n$,
\[
\frac{u(x+r z)}{r}\to \mathbf a(x)\,(-\nu_x\cdot z)_+\qquad\text{as }r\to 0,
\quad\text{with }\ \mathbf a(x)\,H(-\nu_x)=\mathbf c.
\]
Apply this with $r=\varepsilon$ and $z=\pm x$:
\[
\frac{u(x+\varepsilon x)}{\varepsilon}\to \mathbf a(x)\,(-\nu_x\cdot x)_+,
\qquad
\frac{u(x-\varepsilon x)}{\varepsilon}\to \mathbf a(x)\,(\nu_x\cdot x)_+.
\]
Therefore
\[
\varphi_\varepsilon(x)
=\frac{u(x+\varepsilon x)-u(x-\varepsilon x)}{2\varepsilon}
\ \longrightarrow\
\frac{\mathbf a(x)}{2}\Big( (-\nu_x\cdot x)_+-(\nu_x\cdot x)_+\Big)
=-\frac{\mathbf a(x)}{2}\,(\nu_x\cdot x),
\]
since for any $s\in\mathbb R$, $(-s)_+-s_+=-s$.
By dominated convergence (recall $|\varphi_\varepsilon|\le C$) and $\mathbf a(x)H(-\nu_x)=\mathbf c$, we obtain
\begin{equation}\label{left 2}
-\mathbf{c} \int_{\partial^*\Omega}  \varphi_\varepsilon \, H(-\nu_x)\, d\mathscr H^{n-1}
\to
\frac{\mathbf{c}^2}{2} \int_{\partial^*\Omega} (\nu_x \cdot x)\, d\mathscr H^{n-1}
\qquad\text{as }\varepsilon\to 0.
\end{equation}

\medskip
\noindent\textit{Step 4: Expand $\nabla\varphi_\varepsilon$ and split the RHS.}
We compute
\begin{align*}
\nabla\varphi_\varepsilon(x)
&= \frac{(1+\varepsilon)\nabla u((1+\varepsilon)x)-(1-\varepsilon)\nabla u((1-\varepsilon)x)}{2\varepsilon}\\
&= \frac{\nabla u((1+\varepsilon)x)-\nabla u((1-\varepsilon)x)}{2\varepsilon}
+ \frac{\nabla u((1+\varepsilon)x)+\nabla u((1-\varepsilon)x)}{2}.
\end{align*}
Accordingly,
\begin{align*}
\int_{\Omega} DV(\nabla u)\cdot \nabla\varphi_\varepsilon \, dx
&=
\underbrace{\int_{\Omega} DV(\nabla u(x))\cdot
\frac{\nabla u((1+\varepsilon)x)+\nabla u((1-\varepsilon)x)}{2}\, dx}_{=:R_1(\varepsilon)}\\
&\quad +\underbrace{\int_{\Omega} DV(\nabla u(x))\cdot
\frac{\nabla u((1+\varepsilon)x)-\nabla u((1-\varepsilon)x)}{2\varepsilon}\, dx}_{=:R_2(\varepsilon)}.
\end{align*}

\medskip
\noindent\textit{Step 5: Limit of $R_1(\varepsilon)$.}
Since $\nabla u((1\pm\varepsilon)x)\to \nabla u(x)$ for a.e.\ $x\in\Omega$ and $DV(\nabla u)\in L^2(\Omega)$, dominated convergence yields
\[
R_1(\varepsilon)\to \int_\Omega DV(\nabla u)\cdot \nabla u\,dx.
\]
Testing \eqref{eq:weak set finite per} with $\varphi=u$ (by density in $W^{1,2}_0(\Omega)$) gives
\[
\int_\Omega DV(\nabla u)\cdot \nabla u\,dx = \int_\Omega u\,dx,
\]
hence
\begin{equation}\label{right 1}
R_1(\varepsilon)\to \int_\Omega u\,dx.
\end{equation}

\medskip
\noindent\textit{Step 6: Treat $R_2(\varepsilon)$ via Lemma~\ref{second derivative}.}
{We choose $\varepsilon$ as the sequence taking the liminf in Lemma~\ref{second derivative}.}
By Lemma~\ref{second derivative} (with $\alpha_{\pm\varepsilon}(x)=\nabla u((1\pm\varepsilon)x)$) and the {weak definition of divergence}, we have
\begin{align*}
R_2(\varepsilon)
&=
\int_{\Omega} \nabla u(x)\cdot
\frac{DV(\nabla u((1+\varepsilon)x))-DV(\nabla u((1-\varepsilon)x))}{2\varepsilon}\,dx
+o(1)\\
&=
-\int_{\Omega} u(x)\,
\frac{{\rm div}\!\big(DV(\nabla u((1+\varepsilon)x))\big)-{\rm div}\!\big(DV(\nabla u((1-\varepsilon)x))\big)}{2\varepsilon}\,dx
+o(1).
\end{align*}
where {we used $u\in W^{1,2}_0(\Omega)$ as a test function.}
Since 
$${\rm div}(DV(\nabla u((1\pm\varepsilon)x)))=(1\pm\varepsilon)\,\Delta_Hu((1\pm\varepsilon)x),$$ 
we get
\begin{equation}\label{eq:R2-main}
R_2(\varepsilon)
=
-\int_{\Omega} u(x)\,
\frac{(1+\varepsilon)\Delta_H u((1+\varepsilon)x)-(1-\varepsilon)\Delta_H u((1-\varepsilon)x)}{2\varepsilon}\,dx
+o(1).
\end{equation}

\medskip
\noindent\textit{Step 7: Compute the integral in \eqref{eq:R2-main}.}
Set
\[
J(\varepsilon):=\int_{\Omega} u(x)\,
\frac{(1+\varepsilon)\Delta_H u((1+\varepsilon)x)-(1-\varepsilon)\Delta_H u((1-\varepsilon)x)}{2\varepsilon}\,dx.
\]
Using \eqref{eq:weak set finite per}, i.e.
\[
\Delta_H u=\mathbf{c}H(-\nu)\,\mathscr{H}^{n-1}\llcorner \partial^*\Omega - \mathbf{1}_{\Omega}\,dx,
\]
a change of variables gives
\begin{align*}
J(\varepsilon)
=&\ \frac{\mathbf c}{2\varepsilon}\Bigg[
{\int_{(1+\varepsilon)^{-1}\partial^*\Omega} u(x)\,H(-\nu_{(1+\varepsilon)x})\, d\mathscr H^{n-1}(x)}\\
&\qquad \qquad \qquad \qquad \qquad 
{-\int_{(1-\varepsilon)^{-1}\partial^*\Omega} u(x)\,H(-\nu_{(1-\varepsilon)x})\, d\mathscr H^{n-1}(x)}
\Bigg]\\
&\quad +\frac{1}{2\varepsilon}\Bigg[
-\int_{(1+\varepsilon)^{-1}\Omega}(1+\varepsilon)\,u(x)\,dx
+\int_{(1-\varepsilon)^{-1}\Omega}(1-\varepsilon)\,u(x)\,dx
\Bigg].
\end{align*}

\smallskip
\noindent\emph{Boundary part.}
Since dilation preserves normals and $d\mathscr H^{n-1}$ scales by $(1\pm\varepsilon)^{-(n-1)}$, we can rewrite the boundary term as
\begin{multline*}
\frac{\mathbf c}{2\varepsilon}\Bigg[
{\int_{(1+\varepsilon)^{-1}\partial^*\Omega} u(x)\,H(-\nu_{(1+\varepsilon)x})\, d\mathscr H^{n-1}(x)
-\int_{(1-\varepsilon)^{-1}\partial^*\Omega} u(x)\,H(-\nu_{(1-\varepsilon)x})\, d\mathscr H^{n-1}(x)}
\Bigg]\\
=\ {\mathbf c\int_{\partial^*\Omega}\psi_\varepsilon(x)\,H(-\nu_x)\,d\mathscr H^{n-1}(x),}
\end{multline*}
where we set
\begin{equation}\label{def:psi-eps}
\psi_\varepsilon(x):=
{\frac{(1+\varepsilon)^{1-n}u((1+\varepsilon)^{-1}x)-(1-\varepsilon)^{1-n}u((1-\varepsilon)^{-1}x)}{2\varepsilon}.}
\end{equation}

\smallskip
\noindent\emph{Volume part.}
{Since $u\equiv0$ a.e.\ outside $\Omega$, set}
\[
{
A_\varepsilon^+:=\Omega\setminus (1+\varepsilon)^{-1}\Omega,\qquad
A_\varepsilon^-:=\Omega\setminus (1-\varepsilon)^{-1}\Omega .
}
\]
{Then}
\[
{
\int_{(1+\varepsilon)^{-1}\Omega}u\,dx=\int_\Omega u\,dx-\int_{A_\varepsilon^+}u\,dx,\qquad
\int_{(1-\varepsilon)^{-1}\Omega}u\,dx=\int_\Omega u\,dx-\int_{A_\varepsilon^-}u\,dx .
}
\]
{If $x\in A_\varepsilon^\pm$, then $x\in\Omega$ while $(1\pm\varepsilon)x\notin\Omega$. Hence the segment joining
$x$ to $(1\pm\varepsilon)x$ meets $\partial\Omega$ and, since $|x|\le\diam(\Omega)$,}
\[
{
\dist(x,\partial\Omega)\le \varepsilon |x|\le \varepsilon\,\diam(\Omega).
}
\]
{Thus $A_\varepsilon^\pm\subset\mathcal N_\varepsilon(\partial\Omega)$. By the boundary linear growth
$u(x)\le C\,\dist(x,\partial\Omega)$ and the finite upper Minkowski-content bound,}
\[
{
\int_{A_\varepsilon^\pm}u\,dx\le C\varepsilon\,|\mathcal N_\varepsilon(\partial\Omega)|=O(\varepsilon^2).
}
\]
{Consequently the volume part equals}
\begin{align*}
&\frac{1}{2\varepsilon}\Bigg[
-(1+\varepsilon)\int_{(1+\varepsilon)^{-1}\Omega}u\,dx
+(1-\varepsilon)\int_{(1-\varepsilon)^{-1}\Omega}u\,dx
\Bigg]\\
&\qquad= -\int_\Omega u\,dx
+\frac{(1+\varepsilon)\int_{A_\varepsilon^+}u\,dx
-(1-\varepsilon)\int_{A_\varepsilon^-}u\,dx}{2\varepsilon}
=-\int_\Omega u\,dx+o(1).
\end{align*}

\smallskip
\noindent\emph{Putting boundary + volume together.}
We have shown
\begin{equation}\label{eq:Jeps}
J(\varepsilon)=\mathbf c\int_{\partial^*\Omega}\psi_\varepsilon\,H(-\nu)\,d\mathscr H^{n-1}
-\int_\Omega u\,dx + o(1).
\end{equation}

\medskip
\noindent\textit{Step 8: Boundary limit for $\psi_\varepsilon$.}
Fix $x\in\partial^*\Omega$. Note that
\[
(1+\varepsilon)^{-1}x = x-\frac{\varepsilon}{1+\varepsilon}x,\qquad
(1-\varepsilon)^{-1}x = x+\frac{\varepsilon}{1-\varepsilon}x.
\]
Applying Lemma~\ref{basic property u}(4) with $r=\frac{\varepsilon}{1\pm\varepsilon}$ and $z=\mp x$, we obtain
\[
\frac{u((1+\varepsilon)^{-1}x)}{\varepsilon}\to \mathbf a(x)\,(\nu_x\cdot x)_+,
\qquad
\frac{u((1-\varepsilon)^{-1}x)}{\varepsilon}\to \mathbf a(x)\,(-\nu_x\cdot x)_+.
\]
Since {$(1\pm\varepsilon)^{1-n}\to 1$}, it follows from \eqref{def:psi-eps} that
\[
\psi_\varepsilon(x)\to
\frac{\mathbf a(x)}{2}\Big((\nu_x\cdot x)_+-(-\nu_x\cdot x)_+\Big)
=\frac{\mathbf a(x)}{2}\,(\nu_x\cdot x).
\]
By dominated convergence and $\mathbf a(x)H(-\nu_x)=\mathbf c$ we conclude
\begin{equation}\label{eq:psi-limit}
\mathbf c\int_{\partial^*\Omega}\psi_\varepsilon\,H(-\nu)\,d\mathscr H^{n-1}
\ \longrightarrow\
\frac{\mathbf c^2}{2}\int_{\partial^*\Omega}(\nu_x\cdot x)\,d\mathscr H^{n-1}.
\end{equation}

\medskip
\noindent\textit{Step 9: Conclude the limit of $R_2(\varepsilon)$.}
Combining \eqref{eq:R2-main}, \eqref{eq:Jeps}, and \eqref{eq:psi-limit} gives
\begin{equation}\label{right 2}
R_2(\varepsilon)\to
-\frac{\mathbf c^2}{2}\int_{\partial^*\Omega}(\nu_x\cdot x)\,d\mathscr H^{n-1}
+\int_\Omega u\,dx.
\end{equation}

\medskip
\noindent\textit{Step 10: Finish the identity.}
{Taking $\varepsilon\to 0$ along the sequence chosen in Step~6} in \eqref{eq:test vol} and using \eqref{left 1}, \eqref{left 2}, \eqref{right 1}, and \eqref{right 2},
we obtain
\[
-n\int_{\Omega} u\, dx+\frac{\mathbf{c}^2}{2} \int_{\partial^*\Omega} (\nu_x\cdot x) \, d\mathscr H^{n-1}
=
2\int_{\Omega} u\, dx-\frac{\mathbf {c}^2}{2} \int_{\partial^*\Omega} (\nu_x\cdot x) \, d\mathscr H^{n-1}.
\]
Therefore,
\[
(n+2)\int_\Omega u\,dx=\mathbf c^2\int_{\partial^*\Omega}(\nu_x\cdot x)\,d\mathscr H^{n-1}.
\]
Finally, by the divergence theorem for sets of finite perimeter (with $\nu_x$ the outer normal),
\[
\int_{\partial^*\Omega}(\nu_x\cdot x)\,d\mathscr H^{n-1}=\int_\Omega {\rm div}(x)\,dx=n|\Omega|.
\]
This proves \eqref{volume}.
\end{proof}

\section{Proof of Theorem~\ref{main thm}}
\label{sec:proof}

By Lemma~\ref{basic property u}(3), we may assume without loss of generality that $\Omega$ is open (replacing it, if needed, by the
open representative $\mathring\Omega=\{u>0\}$).
Set
\[
\mathcal A(x):=D^2V(\nabla u(x)) \qquad \text{for a.e.\ }x\in\Omega,
\]
{recall that}
$$
{\mathcal A(x) Du= D^2V(\nabla u(x)) \nabla u = DV(\nabla u),}
$$
and define the linearized operator
\[
L_{\mathcal A}:={\rm div}(\mathcal A\nabla\cdot).
\]
Since $K$ is uniformly convex and $\partial K\in C^{2,\gamma}$, the matrix $D^2V$ is bounded and uniformly {positive-definite} on
$\mathbb S^{n-1}$; hence $\mathcal A(\cdot)$ is bounded and uniformly elliptic in $\Omega$ (defining $\mathcal A$ arbitrarily
on the negligible set $\{\nabla u=0\}$ if needed).

\subsection{Green function and harmonic measure for $L_{\mathcal A}$}

Let $\{\Omega_k\}_{k\in\mathbb N}$ be an exhausting sequence of smooth open sets with
\[
\Omega_k\subset\subset\Omega_{k+1}\subset\subset\Omega,
\qquad \bigcup_{k}\Omega_k=\Omega.
\]
Fix $x\in\Omega$. For $k$ large enough so that $x\in\Omega_k$, let $G_{x,k}$ be the unique weak solution of
\[
\left\{
\begin{array}{ll}
L_{\mathcal A} G_{x,k}=-\delta_x,& \text{in }\Omega_k,\\[2pt]
G_{x,k}=0,& \text{a.e.\ in }\mathbb R^n\setminus \Omega_k.
\end{array}
\right.
\]
By the maximum principle, $\{G_{x,k}\}$ is monotonically increasing in $k$, hence we can define
\[
G_x:=\lim_{k\to\infty}G_{x,k}\qquad\text{pointwise in }\Omega\setminus\{x\}.
\]

{
\begin{lem}\label{lem:elliptic-exhaustion}
Let $G_x$ be the Green function constructed above. Then there exists a probability measure $\omega_x$, supported on
$\partial\Omega$, such that, after extending $G_x$ by $0$ outside $\Omega$,
\[
L_{\mathcal A}G_x=\omega_x-\delta_x
\qquad\text{in }\mathscr D'(\mathbb R^n).
\]
Also, the following boundary maximum principle holds:
if $W \geq 0$ in a neighborhood of $\partial\Omega$, $W=0$ outside $\Omega$, and
$L_{\mathcal A}W=f\,\mathbf 1_\Omega\,dx+\mu$ with $f\in L^\infty$ and $\mu$ supported on $\partial\Omega$, then
$\mu\ge0$.
\end{lem}

\begin{proof}
For smooth $\Omega_k$, the zero extension of $G_{x,k}$ satisfies
$L_{\mathcal A}G_{x,k}=\omega_{x,k}-\delta_x$ in distributions, and $\omega_{x,k}$ is a probability measure on
$\partial\Omega_k$. Letting $k\to\infty$ gives the first assertion, while the boundary maximum principle follows
from the maximum principle on $\Omega_k$.
\end{proof}
}

\begin{lem}\label{greens function}
Let $\Omega$ satisfy the assumptions in Theorem~\ref{main thm} and let $u$ solve \eqref{eq:weak set finite per}.
Fix $x\in\Omega$ and let $G_x$ be the Green function constructed above. Then:
\begin{enumerate}
\item[(1)] There exist $\rho=\rho(x,\Omega)>0$ and $M=M(x,\Omega)>0$ such that
\begin{equation}\label{eq:G u}
0\le G_x \le M u \qquad \text{in }\Omega\setminus B_\rho(x).
\end{equation}
In particular, since $u$ is Lipschitz and vanishes on $\partial\Omega$ in the trace sense,
\[
G_x(y)\le C\,\dist(y,\partial\Omega)\qquad \forall\,y\in \Omega\setminus B_\rho(x).
\]

\item[(2)] Extending $G_x$ by $0$ on $\mathbb R^n\setminus\Omega$, the distribution $L_{\mathcal A}G_x$ is a Radon measure
supported on $\partial^*\Omega\cup\{x\}$. Moreover, there exists a bounded Borel function
$\alpha_x:\partial^*\Omega\to[0,\infty)$ such that
\begin{equation}\label{eq:green-measure}
L_{\mathcal A}G_x=\alpha_x\,\mathscr H^{n-1}\llcorner \partial^*\Omega-\delta_x
\qquad\text{in }\mathscr D'(\mathbb R^n),
\end{equation}
and
\begin{equation}\label{eq:alpha-mass}
\int_{\partial^*\Omega}\alpha_x\,d\mathscr H^{n-1}=1.
\end{equation}
\end{enumerate}
\end{lem}

\begin{proof}
\emph{Step 1: Comparison with $u$.}
Choose $\rho=\rho(x,\Omega)>0$ so that $\overline{B_\rho(x)}\subset\subset\Omega$.
By the strong maximum principle, $u>0$ in $\Omega$, hence $u$ has a positive minimum on $\partial B_\rho(x)$.
Since $G_{x,k}$ is smooth in $\Omega_k\setminus\{x\}$ and $G_{x,k}\to 0$ on $\partial\Omega_k$, we may choose
$M=M(x,\Omega)>0$ so large that for all $k$ large enough,
\[
G_{x,k}<Mu \quad \text{on }\partial B_\rho(x)\quad\text{and}\quad G_{x,k}\le Mu \quad\text{on }\partial\Omega_k.
\]
Inside $\Omega_k\setminus B_\rho(x)$ we have
\[
L_{\mathcal A}(Mu-G_{x,k})=M L_{\mathcal A}u - L_{\mathcal A}G_{x,k}=M(-1)-0<0,
\]
so by the maximum principle $G_{x,k}\le Mu$ in $\Omega_k\setminus B_\rho(x)$.
Letting $k\to\infty$ gives \eqref{eq:G u}.

\smallskip
\emph{Step 2: Measure representation.}
{Extend $G_x$ by $0$ on $\mathbb R^n\setminus\Omega$. By Lemma~\ref{lem:elliptic-exhaustion}, the
boundary measures $\omega_{x,k}$ associated with the smooth exhaustion converge weakly-* to a probability measure
$\omega_x$ and}
\[
{
L_{\mathcal A}G_x=\omega_x-\delta_x
\qquad\text{in }\mathscr D'(\mathbb R^n),
}
\]
{where $\omega_x$ is a probability measure supported on $\partial\Omega$. It remains to show that this boundary
measure is carried by $\partial^*\Omega$.}

{Let $\varphi\ge0$ be supported in a small neighborhood of $\partial\Omega$ disjoint from $B_\rho(x)$, and set
$w=Mu-G_x$. By Step~1, $w\ge0$ near $\partial\Omega$, $w=0$ on $\mathbb R^n\setminus\Omega$, and
$L_{\mathcal A}w=-M$ in $\Omega\setminus B_\rho(x)$. Applying 
Lemma~\ref{lem:elliptic-exhaustion} to $w$ gives}
\[
{
0\le \left\langle L_{\mathcal A}w+M\mathbf 1_\Omega\,dx,\varphi\right\rangle
=M\mathbf c\int_{\partial^*\Omega}\varphi\,H(-\nu)\,d\mathscr H^{n-1}
-\int \varphi\,d\omega_x .
}
\]
{Hence, on $\mathbb R^n\setminus B_\rho(x)$,}
\[
{
0\le \omega_x\le M\mathbf c\,H(-\nu)\,\mathscr H^{n-1}\llcorner\partial^*\Omega .
}
\]
{Since $\rho$ is arbitrary, the boundary measure $\omega_x$ is carried by
$\partial^*\Omega$ and has a bounded density. Thus
$\omega_x=\alpha_x\,\mathscr H^{n-1}\llcorner\partial^*\Omega$ for some bounded $\alpha_x\ge0$, which is
\eqref{eq:green-measure}.}

Finally, testing \eqref{eq:green-measure} against the constant function $1$ gives
\[
0=\langle L_{\mathcal A}G_x,1\rangle=\int_{\partial^*\Omega}\alpha_x\,d\mathscr H^{n-1}-1,
\]
which proves \eqref{eq:alpha-mass}.
\end{proof}

\begin{rem}\label{rem:repr}
{Fix $x\in\Omega$ and let $\omega_x:=\alpha_x\,\mathscr H^{n-1}\llcorner\partial^*\Omega$ be the boundary
measure in Lemma~\ref{greens function}. We shall use the trace form of the elliptic-measure representation only for
the directional derivatives appearing below. Namely, if $e\in K$ and $v_e:=\partial_e u$, then
$L_{\mathcal A}v_e=0$ in $\Omega$, $v_e$ is bounded, and Lemma~\ref{basic property u}(4) gives an approximate interior
trace $v_e^*$ for $\mathscr H^{n-1}$-a.e.\ point of $\partial^*\Omega$. Hence, applying \cite[Lemma~3.3]{FZ2025}
to the operator $L_{\mathcal A}$,}
\[
{
v_e(x)=\int_{\partial^*\Omega} v_e^*(y)\,d\omega_x(y)
=\int_{\partial^*\Omega} v_e^*(y)\,\alpha_x(y)\,d\mathscr H^{n-1}(y).
}
\]
\end{rem}

\subsection{A maximum principle for $H(\nabla u)$}

\begin{prop}\label{max nabla u}
Let $\Omega$ be an open bounded set satisfying the assumptions in Theorem~\ref{main thm}, and let $u$ solve
\eqref{eq:weak set finite per}. Then
\[
{\operatorname*{ess\,sup}_{\Omega} H(\nabla u)\le \mathbf{c}.}
\]
\end{prop}

\begin{proof}
Fix $e\in K$ (equivalently, $H_*(e)\le 1$). Since $u\in W^{2,2}_{\loc}(\Omega)$ and
${\rm div}(DV(\nabla u))=-1$ in $\Omega$, we may differentiate the equation in distributions as follows.
Take $\psi\in C_c^2(\Omega)$ and test \eqref{eq:weak set finite per} with $\varphi=\partial_e\psi$; since
$\int_\Omega \partial_e\psi\,dx=0$, we obtain
\[
0=\int_\Omega DV(\nabla u)\cdot \nabla(\partial_e\psi)\,dx
= -\int_\Omega \partial_e\!\bigl(DV(\nabla u)\bigr)\cdot \nabla\psi\,dx.
\]
Using $\partial_e(DV(\nabla u))=D^2V(\nabla u)\,\nabla(\partial_e u)=\mathcal A\,\nabla(\partial_e u)$ a.e., we conclude that
$v_e:=\partial_e u$ is a weak solution of
\[
L_{\mathcal A} v_e = 0\qquad \text{in }\Omega.
\]
Fix $x\in\Omega$ and let $\alpha_x$ be given by Lemma~\ref{greens function}(2).
Applying Remark~\ref{rem:repr} to the $L_{\mathcal A}$-harmonic function $v_e=\partial_e u$, we obtain
\begin{equation}\label{eq:ve-rep}
v_e(x)=\int_{\partial^*\Omega} v_e^*(y)\,\alpha_x(y)\,d\mathscr H^{n-1}(y),
\end{equation}
where $v_e^*$ denotes the approximate boundary trace on $\partial^*\Omega$.
By Lemma~\ref{basic property u}(4), for $\mathscr H^{n-1}$-a.e.\ $y\in\partial^*\Omega$ the interior gradient has the approximate
limit $\nabla u(y)=-\mathbf a(y)\nu_y$, hence
\[
v_e^*(y)=\nabla u(y)\cdot e=-\mathbf a(y)\,\nu_y\cdot e.
\]
Since $e\in K$, we have $(-\nu_y)\cdot e\le \sup_{z\in K}(-\nu_y)\cdot z = H(-\nu_y)$, and Lemma~\ref{basic property u}(4) gives
$\mathbf a(y)H(-\nu_y)=\mathbf c$. Hence
\[
v_e^*(y)\le \mathbf a(y)\,H(-\nu_y)=\mathbf c\qquad\text{for $\mathscr H^{n-1}$-a.e.\ }y\in\partial^*\Omega.
\]
Plugging this into \eqref{eq:ve-rep} and using \eqref{eq:alpha-mass} yields
\[
v_e(x)\le \mathbf c\int_{\partial^*\Omega}\alpha_x\,d\mathscr H^{n-1}=\mathbf c.
\]
Since $x\in\Omega$ and $e\in K$ were arbitrary, we conclude
\[
{\nabla u(x)\cdot e\le \mathbf c\quad \forall\,e\in K\quad\text{for a.e. }x\in\Omega,}
\]
and taking the supremum over $e\in K$ gives {$H(\nabla u)\le \mathbf c$ a.e.\ in $\Omega$}.
\end{proof}

\subsection{Conclusion via the $P$-function}

\begin{proof}[Proof of Theorem~\ref{main thm}]
By Lemma~\ref{basic property u}(3), we may assume that $\Omega$ is open. Define the $P$-function
\[
P(x)=H(\nabla u(x))^2+\frac{2}{n}u(x)=2V(\nabla u(x))+\frac{2}{n}u(x)\qquad\text{for a.e.\ }x\in\Omega.
\]

\smallskip
\noindent\emph{Step 1: Subharmonicity of $P$.}
{Since $u\in W^{2,2}_{\loc}(\Omega)$ and}
$$
{L_{\mathcal A}u={\rm div}(\mathcal A\nabla u)={\rm div}(DV(\nabla u))=-1 
\quad \text{in } \ \Omega,}
$$  
{by differentiating the equation above and applying the $0$-homogeneity of $D^2V$,  we obtain the following Weinberger-type identity}
\begin{equation}\label{eq:LA-P}
L_{\mathcal A}\bigl(V(\nabla u)\bigr)=\bigl|\,\mathcal A^{1/2}D^2u\,\mathcal A^{1/2}\bigr|^2
\qquad\text{a.e.\ in }\Omega,
\end{equation}
hence
\[
L_{\mathcal A}P
=2\bigl|\,\mathcal A^{1/2}D^2u\,\mathcal A^{1/2}\bigr|^2-\frac{2}{n}.
\]
At each point, set $B:=\mathcal A^{1/2}D^2u\,\mathcal A^{1/2}$, which is symmetric. Then
\[
{\rm tr}(B)={\rm tr}(\mathcal A D^2u)=\Delta_H u=-1,
\]
and by Cauchy--Schwarz,
\[
|B|^2 \ge \frac{1}{n}\,{\rm tr}(B)^2=\frac{1}{n}.
\]
Therefore $L_{\mathcal A}P\ge 0$ in $\Omega$ in the weak sense.

\smallskip
\noindent\emph{Step 2: Maximum principle and sharp bound for $P$.}
Fix $\eta>0$ and consider the superlevel set $\{u\ge \eta\}\subset\subset \Omega$. Since $P$ is $L_{\mathcal A}$-subharmonic,
the weak maximum principle yields
\[
\esssup_{\{u\ge \eta\}}P=\esssup_{\partial\{u\ge \eta\}}P.
\]
On $\partial\{u\ge \eta\}$ we have $u=\eta$, and by Proposition~\ref{max nabla u},
$H(\nabla u)\le \mathbf c$. Hence
\[
P \le \mathbf c^2+\frac{2}{n}\eta \qquad\text{on }\partial\{u\ge \eta\}.
\]
Letting $\eta\downarrow 0$ gives
\begin{equation}\label{eq:P-bound}
P\le \mathbf c^2 \qquad\text{a.e.\ in }\Omega.
\end{equation}

\smallskip
\noindent\emph{Step 3: Rigidity from the volume identity.}
Using ${\rm div}(DV(\nabla u))=-1$ in $\Omega$ and $u\in W^{1,2}_0(\Omega)$, we have
\[
\int_\Omega u\,dx=\int_\Omega DV(\nabla u)\cdot \nabla u\,dx
{= \int_\Omega 2V(\nabla u)\,dx,}
\]
{where the $2$-homogeneity of $V$ was applied in the second identity.} 
Therefore,
\[
\int_\Omega P\,dx
=\int_\Omega\left(2V(\nabla u)+\frac{2}{n}u\right)\,dx
=\left(1+\frac{2}{n}\right)\int_\Omega u\,dx
=\frac{n+2}{n}\int_\Omega u\,dx.
\]
By Lemma~\ref{lem:volume}, $\frac{n+2}{n}\int_\Omega u\,dx=\mathbf c^2|\Omega|$, hence
\[
\int_\Omega P\,dx=\mathbf c^2|\Omega|.
\]
Together with \eqref{eq:P-bound} this forces
\[
P\equiv \mathbf c^2 \qquad\text{a.e.\ in }\Omega,
\]
and consequently $L_{\mathcal A}P\equiv 0$. In particular, equality holds in the Cauchy--Schwarz step, so
$B=\mathcal A^{1/2}D^2u\,\mathcal A^{1/2}$ must be a scalar multiple of the identity matrix, namely
\[
\mathcal A^{1/2}D^2u\,\mathcal A^{1/2}=-\frac{1}{n}I,
\]
or equivalently,
\[
D(DV(\nabla u))=\mathcal A D^2u=-\frac{1}{n}I \qquad\text{a.e.\ in }\Omega.
\]
{The identity above implies, exactly as in
\cite[(5.17)--(5.19)]{CS2009}, that on each connected component $\Omega_i$ of the open representative
$\Omega^+=\{u>0\}$ there is a point $x_i$ such that}
\[
{
DV(\nabla u(x))=-\frac{x-x_i}{n}\qquad\text{for a.e.\ }x\in\Omega_i .
}
\]
{Since $V_*(x):=\frac12 H_*^2(x)$ is the Legendre dual of $V$, this is equivalent to}
\[
{
\nabla u(x)=\frac1n\,DV_*\bigl(-(x-x_i)\bigr),
}
\]
{and therefore, on $\Omega_i$,}
\[
{
u(x)=\frac{r_i^2-H_*^2\bigl(-(x-x_i)\bigr)}{2n}
}
\]
{for some $r_i>0$. Hence each component is $x_i-r_iK$. 
Since $\Omega$ is indecomposable this proves that, up to a translation, $\Omega=-rK$ and}
\[
{
u(x)=\frac{r^2-H_*^2(-x)}{2n}
}
\]
{for some $r>0$.}
\end{proof}

\appendix
\section{A covering lemma under ADR with a $1/|\log r|$ gain at selected scales}

Let $\Omega\subset\mathbb R^n$ be a bounded set of finite perimeter, with reduced boundary $\partial^*\Omega$.
For $r>0$ and a set $E\subset\mathbb R^n$, we write
\[
\Ncal_r(E):=\{z\in\mathbb R^n:\ \dist(z,E)\le r\}.
\]
We also denote by $\omega_{n-1}$ the $(n-1)$-dimensional Hausdorff measure of the unit ball in a linear hyperplane,
i.e.\ $\omega_{n-1}=\mathscr H^{n-1}(B_1\cap P)$ for any linear hyperplane $P$.

For $x\in\partial^*\Omega$ and $r>0$ define the $\beta$-number
\begin{equation}\label{def:beta}
  \beta(x,r):=\inf_{P}\; r^{1-n}\!\!\int_{\partial^*\Omega\cap B(x,r)} \frac{\dist(z,P)}{r}\,d\mathscr H^{n-1}(z),
\end{equation}
where the infimum runs over all affine hyperplanes $P\subset\mathbb R^n$.

Assume the global $\beta$-square bound
\begin{equation}\label{ass:A1}
  \int_{\partial^*\Omega}\int_0^1 \beta(x,s)^2\,\frac{ds}{s}\,d\mathscr H^{n-1}(x)\ \le\ A_1<\infty.
\end{equation}

\begin{defn}[Ahlfors--David regularity (ADR)]\label{def:ADR}
We say that $\partial^*\Omega$ is \emph{Ahlfors--David regular} if there exist $0<r_{\rm ADR}\le 1$ and constants
$0<c_{\rm ADR}\le C_{\rm ADR}<\infty$ such that
\begin{equation}\label{eq:ADR}
  c_{\rm ADR}\,r^{n-1}\ \le\ \mathscr H^{n-1}(\partial^*\Omega\cap B(x,r))\ \le\ C_{\rm ADR}\,r^{n-1}
  \quad \forall\,x\in\partial^*\Omega,\ \ 0<r<r_{\rm ADR}.
\end{equation}
\end{defn}

\begin{prop}[ADR covering with density control and $1/|\log r|$ gain on $\beta$-bad]\label{prop:ADR}
Fix $\eta\in(0,1)$ and let $\theta:(0,1)\to(0,1)$ be any function with $\theta(\eta)\downarrow 0$ as $\eta\downarrow 0$.
Then there exist $r_0=r_0(\Omega,\eta,{\rm ADR})\in(0,r_{\rm ADR})$ and a Borel set $F_\eta\subset\partial^*\Omega$ such that, for all $x \in F_\eta$ and $0<s\leq r_0$,
\begin{equation}\label{eq:Egorov}
  \mathscr H^{n-1}(\partial^*\Omega\setminus F_\eta)\ \le\ \theta(\eta)\,\eta^{\,n-1},\qquad
  \frac{\mathscr H^{n-1}(\partial^*\Omega\cap B(x,s))}{\omega_{n-1}s^{n-1}}\in\Big(1-\tfrac{\eta}{10},\,1+\tfrac{\eta}{10}\Big).
\end{equation}
{Moreover, let $G_\eta\subset\partial^*\Omega$ be any Borel set satisfying}
\[
{
\mathscr H^{n-1}(\partial^*\Omega\setminus G_\eta)\le \theta(\eta)\,\eta^{n-1}.
}
\]
{Then there exists a sequence $r_j\downarrow0$ such that, for each $j$, there is a Besicovitch covering}
\[
{
\mathcal B_{r_j}=\mathcal G_{r_j}\cup\mathcal R_{r_j}
}
\]
{of $\Ncal_{r_j/2}(\partial\Omega)$ by balls $B(x_h,r_j)$ with centers in $\partial^*\Omega$ and bounded overlap, with the following properties:}

\begin{enumerate}
\item\label{item:gooddensity}
{If $B(x_h,r_j)\in\mathcal G_{r_j}$, then $x_h\in F_\eta\cap G_\eta$, $\beta(x_h,r_j)\le\eta$, and}
\[
{
  \frac{\mathscr H^{n-1}(\partial^*\Omega\cap B(x_h,r_j))}{\omega_{n-1}r_j^{n-1}}\in(1-\eta,\,1+\eta).
}
\]

\item\label{item:badcount}
{The bad subfamily satisfies}
\[
{
  \#\mathcal R_{r_j}\ \le\ C(n,{\rm ADR})\,r_j^{-(n-1)}
  \left(\frac{\eta^{-2}A_1}{|\log r_j|}+\theta(\eta)\right).
}
\]
\end{enumerate}
\end{prop}

\begin{proof}[Proof of Proposition~\ref{prop:ADR}]
\noindent\emph{Step 1: An Egorov set for density.}
For $x\in\partial^*\Omega$ set
\[
f_r(x):=\frac{\mathscr H^{n-1}(\partial^*\Omega\cap B(x,r))}{\omega_{n-1}r^{n-1}}.
\]
By De Giorgi's structure theorem \cite[Theorem~15.9]{M2012}, $f_s(x)\to 1$ as $s\downarrow 0$ for
$\mathscr H^{n-1}$-a.e.\ $x\in\partial^*\Omega$.
Fix $\eta\in(0,1)$ and define
\[
A_m:=\left\{x\in\partial^*\Omega:\ \sup_{0<s\le 1/m}|f_s(x)-1|\le \tfrac{\eta}{10}\right\},\qquad m\in\mathbb N.
\]
Then $A_m\uparrow \partial^*\Omega$ up to a null set, hence
$\mathscr H^{n-1}(\partial^*\Omega\setminus A_m)\downarrow 0$ as $m\to\infty$.
Choose $m$ so large that
\[
\mathscr H^{n-1}(\partial^*\Omega\setminus A_m)\ \le\ \theta(\eta)\,\eta^{\,n-1}.
\]
Set $F_\eta:=A_m$ and $r_0:=1/m$. This gives \eqref{eq:Egorov}.

\smallskip
\noindent\emph{Step 2: Pointwise control of the number of $\beta$-bad centers.}
Fix $s\in(0,r_0)$ and let $\{x_h\}\subset\partial^*\Omega$ be any $s/5$-separated set.
Write $S_s:=\{h:\ \beta(x_h,s)>\eta\}$.
If $z\in B(x_h,s/10)$ then $B(x_h,s)\subset B(z,2s)$, and therefore
\begin{equation}\label{eq:beta-compare}
  \beta(x_h,s)\ \le\ 2^{n}\,\beta(z,2s).
\end{equation}
The balls $\{B(x_h,s/10)\}_{h\in S_s}$ are disjoint; by the lower ADR bound,
\[
\mathscr H^{n-1}(\partial^*\Omega\cap B(x_h,s/10))\ge c_{\rm ADR}(s/10)^{n-1}.
\]
Hence,
\[
\begin{aligned}
\eta^2\,\#S_s
&\le \sum_{h\in S_s}\beta(x_h,s)^2\\
&\le \sum_{h\in S_s}\frac{2^{2n}}{\mathscr H^{n-1}(\partial^*\Omega\cap B(x_h,s/10))}
   \int_{\partial^*\Omega\cap B(x_h,s/10)}\beta(z,2s)^2\,d\mathscr H^{n-1}(z)\\
&\le C(n,{\rm ADR})\,s^{-(n-1)}\int_{\partial^*\Omega}\beta(z,2s)^2\,d\mathscr H^{n-1}(z).
\end{aligned}
\]
Taking the supremum over all $s/5$-separated families, this gives
\begin{equation}\label{eq:Nbeta-point}
  N_\beta(s)\ \le\ C(n,{\rm ADR})\,s^{-(n-1)}\,\eta^{-2}\!\int_{\partial^*\Omega}\beta(z,2s)^2\,d\mathscr H^{n-1}(z),
\end{equation}
where $N_\beta(s)$ denotes that supremum.

\smallskip
\noindent\emph{Step 3: Selecting scales with a $(1/|\log r|)$-gain.}
For $k\ge 1$, let $I_k=[2^{-k-1},2^{-k}]$. Integrating \eqref{eq:Nbeta-point} on $I_k$ and changing variables $t=2s$, we get
\[
  \int_{I_k} N_\beta(s)\,\frac{ds}{s}
  \ \le\ C(n,{\rm ADR})\,2^{(n-1)k}\,\eta^{-2}\int_{\partial^*\Omega}\int_{2^{-k}}^{2^{-k+1}}\beta(z,t)^2\,\frac{dt}{t}\,d\mathscr H^{n-1}(z).
\]
Define
\[
a_k:=\int_{\partial^*\Omega}\int_{2^{-k}}^{2^{-k+1}}\beta(z,t)^2\,\frac{dt}{t}\,d\mathscr H^{n-1}(z),
\]
so that $\sum_{k\ge1}a_k\le A_1$ by \eqref{ass:A1}.
Given $K\in\mathbb N$ large, choose $k\in\{K/2,\dots,K\}$ such that $a_k\le 2A_1/K$ (by averaging).
Then, by averaging $N_\beta(s)$ over $I_k$ with respect to $ds/s$, there exists $s_k\in I_k$ such that
\[
  N_\beta(s_k)\ \le\ \frac{1}{\log 2}\int_{I_k} N_\beta(s)\,\frac{ds}{s}
  \ \le\ C(n,{\rm ADR})\,2^{(n-1)k}\,\frac{\eta^{-2}A_1}{K}.
\]
Since $s_k\in[2^{-k-1},2^{-k}]$, one has $2^{(n-1)k}\le C\,s_k^{-(n-1)}$ and also $K\sim k\sim |\log s_k|$, hence\footnote{		{This is the elementary scale-selection principle behind the logarithmic gain: on each long dyadic block,
		at least one scale is controlled by the average of the square function on that block.}}
\[
  N_\beta(s_k)\ \le\ C(n,{\rm ADR})\,s_k^{-(n-1)}\,\frac{\eta^{-2}A_1}{|\log s_k|}.
		\]
{Choosing a sequence $K\to\infty$ gives $r_j:=s_k\downarrow0$ with the estimate above at $s=r_j$.}

\smallskip
\noindent\emph{{Step 4: Building the cover and counting the bad subfamily.}}
{Fix one of the selected scales $r=r_j$ and set $F:=F_\eta\cap G_\eta$. Let $Y_r\subset F$ be a maximal $r/5$-separated set. Let}
\[
{
Z_r:=\{x\in\partial^*\Omega:\ \dist(x,Y_r)>r/2\},
}
\]
{and let $W_r\subset Z_r$ be a maximal $r/5$-separated set. Then $Y_r\cup W_r$ is $r/5$-separated and every point of
$\partial^*\Omega$ lies within distance $r/2$ of $Y_r\cup W_r$. Hence the balls $B(x_h,r)$ with
$x_h\in Y_r\cup W_r$ cover $\Ncal_{r/2}(\partial\Omega)$ and have bounded overlap.}

{Declare a ball with center $x_h\in Y_r$ to be good if $\beta(x_h,r)\le\eta$, and put all remaining balls in
$\mathcal R_r$. The good balls satisfy (\ref{item:gooddensity}) by \eqref{eq:Egorov}. It remains to count the bad balls.

The $\beta$-bad centers in $Y_r$ are controlled by the selected-scale estimate from Step~3. If $w\in W_r$, then
$\dist(w,F)>3r/10$; otherwise a point of $F$ within $3r/10$ of $w$ would be within $r/5$ of some point of $Y_r$, contradicting
$w\in Z_r$. Thus}
\[
{
B(w,r/10)\cap\partial^*\Omega\subset \partial^*\Omega\setminus F .
}
\]
{The balls $B(w,r/10)$, $w\in W_r$, are disjoint, and the lower ADR bound gives}
\[
{
\#W_r\,c(n,{\rm ADR})\,r^{n-1}
\le \mathscr H^{n-1}(\partial^*\Omega\setminus F)
\le 2\,\theta(\eta)\,\eta^{n-1}
\le 2\,\theta(\eta).
}
\]
{Combining this estimate with the $\beta$-bad count proves (\ref{item:badcount}).}
\end{proof}

{
\begin{lem}[Geometric covering of the intermediate layer]\label{lem:flat-cover}
Fix $\kappa\in(0,1)$. There exist $\bar\kappa=\bar\kappa(\kappa)\in(0,\kappa)$, $\eta_0=\eta_0(n,\kappa)>0$, and
$C_\kappa=C(n,\kappa)$ with the following property. Set
\[
c_\kappa:=(4+\kappa^{-1}+\bar\kappa^{-1})^{-1}.
\]
Let $0<\eta\le\eta_0$, let $\theta$ be the function used in Proposition~\ref{prop:ADR}, and let
$G_{\kappa,\eta}\subset\partial^*\Omega$ satisfy
\[
\mathscr H^{n-1}(\partial^*\Omega\setminus G_{\kappa,\eta})\le \theta(\eta)\,\eta^{n-1}.
\]
Let $r_{\kappa,\eta}>0$ be such that, for every
$y\in G_{\kappa,\eta}$ and every $0<s\le r_{\kappa,\eta}$,
\[
\beta(y,s)\le\eta,\qquad 
\frac{\mathscr H^{n-1}(\partial^*\Omega\cap B(y,s))}{\omega_{n-1}s^{n-1}}\in(1-\eta,1+\eta).
\]
Let $0<r<r_{\rm ADR}$ with $r\le c_\kappa r_{\kappa,\eta}$, and let
$\{B(y_h,r)\}_h=\mathcal G_r\cup\mathcal R_r$ be a covering given by Proposition~\ref{prop:ADR} with 
$G_\eta=G_{\kappa,\eta}$.
After the normalization $0\in\Omega$ and $\diam(\Omega)=1$, set
\[
E_{\kappa,r}:=\{x\in\Omega:\ \kappa r<\dist(x,\partial\Omega)\le \kappa^{-1}r,\quad
(1+r)x\in\Omega\setminus\mathcal N_{\kappa r}(\partial\Omega)\}.
\]
For $B(y_h,r)\in\mathcal G_r$ set
\[
\widehat B_h:=\widehat B_{y_h,\bar\kappa,r},\qquad
E_h:=\{x\in E_{\kappa,r}:\ x\in \widehat B_h,\ (1+r)x\in \widehat B_h\}.
\]
Then
\[
E_{\kappa,r}\setminus\bigcup_{B(y_h,r)\in\mathcal G_r}E_h
\subset
\bigcup_{B(y_h,r)\in\mathcal R_r} B(y_h,C_\kappa r).
\]
\end{lem}

\begin{proof}
Choose $\bar\kappa=\bar\kappa(\kappa)\in(0,\kappa)$ so small that, for every unit vector $\nu$, the compact set
\[
S_{\kappa,\nu}:=\bigl\{z:\ \kappa\le -z\cdot\nu\le \kappa^{-1}+2,\ |z|\le \kappa^{-1}+2\bigr\}
\]
is compactly contained in
\[
B_{\frac{\bar\kappa^{-1}-\bar\kappa}{2}}\!\left(-\frac{\bar\kappa^{-1}+\bar\kappa}{2}\nu\right).
\]
This is possible because these balls increase to the half-space $\{-z\cdot\nu>0\}$ on compact subsets as
$\bar\kappa\downarrow0$.

Choose $C_\kappa>\kappa^{-1}+2$. We claim that, after decreasing $\eta_0$ if necessary, the conclusion holds with these
constants. If not, we could find
$\eta_i\downarrow0$, admissible sets $G_{\kappa,\eta_i}$ and scales $r_{\kappa,\eta_i}$ as above, radii
$r_i\le c_\kappa r_{\kappa,\eta_i}$, good/bad coverings, and points
$x_i\in E_{\kappa,r_i}$ which stay outside the $C_\kappa r_i$-enlargement of the bad balls but do not belong to any
corresponding $E_h$. Let $y_i\in\partial\Omega$ satisfy
$|x_i-y_i|=\dist(x_i,\partial\Omega)$. Since $x_i\in E_{\kappa,r_i}$, both $x_i$ and
$(1+r_i)x_i$ have distance at least $\kappa r_i$ from the boundary; moreover
\[
\dist(x_i,\partial\Omega)\le \kappa^{-1}r_i,\qquad
\dist((1+r_i)x_i,\partial\Omega)\le(\kappa^{-1}+1)r_i,
\]
because $0\in\Omega$, $\diam(\Omega)=1$, and $x_i\in\Omega$. The covering property gives a center
$y_{h_i}$ with $|y_i-y_{h_i}|\le r_i$. If $B(y_{h_i},r_i)$ were bad, then
$|x_i-y_{h_i}|\le(\kappa^{-1}+1)r_i$, contradicting the choice of $x_i$.
Thus $B(y_{h_i},r_i)$ is good.

Rescale by $r_i$ around $y_{h_i}$. Since $y_{h_i}\in G_{\kappa,\eta_i}$ and
$r_i\le c_\kappa r_{\kappa,\eta_i}$, the defining flatness and density bounds for $G_{\kappa,\eta_i}$ hold at every scale
$s r_i$ with $s\le 4+\kappa^{-1}+\bar\kappa^{-1}$. The standard compactness argument for finite-perimeter sets then gives,
after passing to a subsequence, that the rescaled boundaries converge as Radon measures on these fixed compact sets to a
hyperplane and the rescaled sets converge in
$L^1_{\rm loc}$ to the corresponding half-space with outer normal $\nu$. The rescaled points corresponding to $x_i$ and
$(1+r_i)x_i$ are bounded and, by the distance bounds above and the flat convergence, every limit point
belongs to $S_{\kappa,\nu}$. Moreover, the rescaled interior balls
$\widehat B_{h_i}$ converge to the fixed ball
\[
B_{\frac{\bar\kappa^{-1}-\bar\kappa}{2}}\!\left(-\frac{\bar\kappa^{-1}+\bar\kappa}{2}\nu\right).
\]
Since $S_{\kappa,\nu}$ is compactly contained in this ball, both $x_i$ and $(1+r_i)x_i$ belong to $\widehat B_{h_i}$ for
all large $i$, contradicting
$x_i\notin E_{h_i}$. This proves the lemma.
\end{proof}
}

\end{document}